\documentclass{elsarticle}

\usepackage{amsfonts,amssymb,amsmath,amsthm,synttree}
\usepackage[mathscr]{eucal}

\usepackage{hyperref}

\allowdisplaybreaks[3]

\newtheorem{thrm}{Theorem}[section]
\newtheorem{lem}[thrm]{Lemma}

\theoremstyle{definition}
\newtheorem{definition}[thrm]{Definition}
\newtheorem{defin}[thrm]{Definition}

\newtheorem{mainthrm}[thrm]{Main Theorem}
\numberwithin{equation}{section}

\newcommand{\labeq}[1]{\label{eq:#1}}
\newcommand{\refeq}[1]{(\ref{eq:#1})}
\newcommand{\labt}[1]{\label{thm:#1}}
\newcommand{\reft}[1]{Theorem~\ref{thm:#1}}
\newcommand{\labmt}[1]{\label{thm:#1}}
\newcommand{\refmt}[1]{Main~Theorem~\ref{thm:#1}}
\newcommand{\labl}[1]{\label{lemma:#1}}
\newcommand{\refl}[1]{Lemma~\ref{lemma:#1}}
\newcommand{\labd}[1]{\label{definition:#1}}
\newcommand{\refd}[1]{Definition~\ref{definition:#1}}

\newcommand{\refs}[1]{Section~\ref{section:#1}}
\newcommand{\labs}[1]{\label{section:#1}}

\newcommand{\dimh}[1]{\hbox{$\dim_{\hbox{H}}$}\left( #1\right)}

\newcommand{\NN}{\mathbb{N}_2^{\mathbb{N}}}

\newcommand{\wrtQ}{\hbox{ w.r.t. }Q}

\newcommand{\floor}[1]{\left\lfloor #1 \right\rfloor} 
\newcommand{\ceil}[1]{\left\lceil #1 \right\rceil} 
\newcommand{\br}[1]{\left\{ #1 \right\}}
\newcommand{\abs}[1]{\left| #1 \right|}
\newcommand{\pr}[1]{\left( #1 \right)}

\newcommand{\Ob}[1]{\mathscr{O}_b( #1 )}
\newcommand{\OQ}[1]{\mathscr{O}_Q ( #1 )}
\newcommand{\Obx}{\Ob{x}}
\newcommand{\OQx}{\OQ{x}}
\newcommand{\OQmrx}{\mathscr{O}_{Q,m,r}(x)}
\newcommand{\OQfmrx}[1]{\mathscr{O}_{Q, #1, m, r}(x)}
\newcommand{\OQfx}{\mathscr{O}_{Q,f}(x)}

\newcommand{\OQp}[1]{\mathscr{O}'_Q ( #1 )}
\newcommand{\OQxp}{\OQp{x}}
\newcommand{\OQmrxp}{\mathscr{O}'_{Q,m,r}(x)}
\newcommand{\OQfmrxp}[1]{\mathscr{O}'_{Q, #1, m, r}(x)}
\newcommand{\OQfxp}{\mathscr{O}'_{Q,f}(x)}

\newcommand{\Fpqstart}[4]{ \frac{\#( #1(\mathbb{N}) \bigcap #2(\mathbb{N}) \bigcap \{ #3, #3+1, \cdots #4\})}{\# (#2(\mathbb{N}) \bigcap \{ #3, #3+1, \cdots #4\})}}
\newcommand{\dg}{\rho}
\newcommand{\uD}[2]{\overline{D}_{#1}^{#2}}
\newcommand{\supp}[1]{\text{supp(} #1 \text{)}}
\newcommand{\lcm}{\text{lcm}}

\newcommand{\de}{\delta}

\newcommand{\blank}[1]{ }

\author[D. Airey]{Dylan Airey}
\address[D. Airey]{
Department of Mathematics, University of Texas at Austin, 2515 Speedway, Austin, TX 78712-1202, USA}
\ead{dylan.airey@utexas.edu}

\author[B. Mance]{Bill Mance}
\address[B. Mance]{Department of Mathematics, University of North Texas, General Academics Building 435, 1155 Union Circle,  \#311430, Denton, TX 76203-5017, USA}
\ead{mance@unt.edu}

\begin{document}

\title{Unexpected distribution phenomenon resulting from Cantor series expansions}

\begin{abstract}
We explore in depth the number theoretic and statistical properties of certain sets of numbers arising from their Cantor series expansions.
As a direct consequence of our main theorem we deduce numerous new results
as well as strengthen known ones. 
\end{abstract}

\maketitle

\section{Introduction}

%make note of applications section

We will prove a general result that will have six seemingly unrelated classes of number theoretic applications.  Unfortunately, it will take several pages  to state this result.  After this we describe the applications and then prove our theorem.

The $Q$-Cantor series expansion, first studied by G. Cantor in \cite{Cantor}%\footnote{G. Cantor's motivation to study the Cantor series expansions was to extend the well known proof of the irrationality of the number $e=\sum 1/n!$ to a larger class of numbers.  Results along these lines may be found in the monograph of J. Galambos \cite{Galambos}.}, %See also \cite{TijdemanYuan} and \cite{HT}.  }
is a natural generalization of the $b$-ary expansion. Let $\mathbb{N}_k:=\mathbb{Z} \cap [k,\infty)$.  If $Q \in \NN$, then we say that $Q$ is a \textbf{ basic sequence}.  If $\lim_{n \to \infty} q_n=\infty$, then we say that $Q$ is \textbf{ infinite in limit}.
Given a basic sequence $Q=\{q_n\}_{n=1}^{\infty}$, the \textbf{$Q$-Cantor series expansion} of a real $x$ in $\mathbb{R}$ is the (unique) %\footnote{Uniqueness can be proven in the same way as for the $b$-ary expansions.} 
expansion of the form
\begin{equation} \labeq{cseries}
x=E_0(x)+\sum_{n=1}^{\infty} \frac {E_n(x)} {q_1 q_2 \cdots q_n},
\end{equation}
where $E_0(x)=\floor{x}$ and $E_n(x)$ is in $\{0,1,\cdots,q_n-1\}$ for $n\geq 1$ with $E_n(x) \neq q_n-1$ infinitely often. We will write $E_n$ in place of $E_n(x)$ when there is no room for confusion.  Moreover, we will abbreviate \refeq{cseries} with the notation $x=E_0.E_1E_2E_3\cdots\wrtQ$.
Clearly, the $b$-ary expansion is a special case of \refeq{cseries} where $q_n=b$ for all $n$.  If one thinks of a $b$-ary expansion as representing an outcome of repeatedly rolling a fair $b$-sided die, then a $Q$-Cantor series expansion may be thought of as representing an outcome of rolling a fair $q_1$ sided die, followed by a fair $q_2$ sided die and so on.

The study of normal numbers and other statistical properties of real numbers with respect to large classes of Cantor series expansions was started by P. Erd\H{o}s, A. R\'{e}nyi and P. Tur\'{a}n.
% were the first to study normal numbers with respect to the Cantor series expansions and other related statistical properties of digits of real numb
This early work was done by P. Erd\H{o}s and A. R\'{e}nyi in \cite{ErdosRenyiConvergent} and \cite{ErdosRenyiFurther} and by A. R\'{e}nyi in \cite{RenyiProbability}, \cite{Renyi}, and \cite{RenyiSurvey} and by P. Tur\'{a}n in \cite{Turan}.

%\begin{definition}
We recall the following standard definitions (see \cite{KuN}).
An \textbf{asymptotic distribution function} $f: [0,1] \to [0,1]$ is a non-decreasing function such that $f(0)=0$ and $f(1)=1$.
%\end{definition}
%Define a counting function 
For a sequence of real numbers $\omega = \{x_n\}$ with $x_n \in [0,1)$ and an interval $I \subseteq [0,1]$,  define $A_n(I,\omega) := \# \{i\leq n: x_i \in I \}$.
%\begin{definition}
A sequence of real numbers $\omega = \{x_n\}$ has asymptotic distribution function $f$ if %we have that
\begin{equation*}
\lim_{n\to\infty} \frac{A_n([0,x), \omega)}{n} = f(x). %\hbox{ for all } x \in [0,1].
\end{equation*}
%\end{definition}
For the rest of this paper we will abbreviate asymptotic distribution function as adf.
We say that a sequence $\omega$ is \textbf{uniformly distributed mod $1$} if $\omega$ has $f(x)=x$ as its adf.    For the rest of this paper we will abbreviate uniformly distributed mod 1 as u.d. mod 1.  Clearly, not all sequences have an adf.
%Note that this limit may not always exist, but we can instead define the following.
%\begin{definition}
A sequence of real numbers $\omega = \{x_n\}$ has \textbf{upper asymptotic distribution function} $\overline{f}$ if $$\varlimsup_{n \to \infty} \frac{A_n([0,x),\omega)}{n} = \overline{f}(x). $$
The sequence $\omega$ has \textbf{lower asymptotic distribution function} $\underline{f}$ if $$\liminf_{n \to \infty} \frac{A_n([0,x),\omega)}{n} = \underline{f}(x).$$
%\end{definition}
Every sequence of real numbers $\omega$ has an upper and a lower adf.
We note that the sequence $\omega$ has adf $f(x)$ if and only if $f=\overline{f}=\underline{f}$.  

%For $b \in \mathbb{N}_2$, let $T_b:[0,1) \to [0,1)$

A great deal of information about the $b$-ary expansion of a real number $x$ may be obtained by studying the distributional properties of the sequence $\Obx:=\{b^n x \}_{n=0}^\infty$.  For example, it is well known that a real number $x$ is normal in base $b$ if and only if the sequence $\Obx$ is uniformly distributed mod $1$.

Thus, we are motivated to make the following definitions for the Cantor series expansions.
For every basic sequence $Q$, define $T_{Q,n}(x):=q_n q_{n-1}\cdots q_1 x \pmod{1}$ and $\OQx:=\br{T_{Q,n}(x)}_{n=0}^\infty$. For integers $m$ and $r$, we define $\OQmrx := \br{T_{Q,mn+r}(x)}_{n=0}^\infty$.  For any eventually increasing function $f: \mathbb{N} \to \mathbb{N}$,  we define $\OQfmrx{f} := \br{T_{Q,f(mn+r)}(x)}_{n=0}^\infty$.  Furthermore, set $\OQxp:=\br{\frac {E_n} {q_n}}_{n=1}^\infty$. The sequences $\OQmrxp$ and $\OQfmrxp{f}$ are defined similarly.

\

One of the main goals of this paper that will be resolved in \refmt{Main} will be to show that for a large class of functions $f$ the set of real numbers $x$ where the sequence $\OQfmrx{f}$ has a specified upper and lower adf has full Hausdorff dimension.  Numerous results will follow from this fact.  We will need to give several definitions to fully describe the scope of this theorem.

\

We should note that the relationship between the digits of the $Q$-Cantor series expansion of a real number $x$ and the sequence $\OQx$ is far more complex than the analogous relationship for $b$-ary expansions.  The most current results can be found in \cite{ppq1}.  Thus, when generalizing problems involving digits in some $b$-ary expansion, we can consider either a problem involving digits in a $Q$-Cantor series expansion or a problem involving the distributional properties of the sequence $\OQx$.  Often the theory will be different.  For this paper we will always choose the latter option.

The sequence $\OQx$ was studied by J. Galambos in \cite{Galambos2} and by T. \u{S}al\'{a}t in \cite{Salat} and several other papers. The main focus of this paper is to study %the Hausdorff dimension of 
sets of reals numbers $x$ so that $\OQx$ and various subsequences of $\OQx$ have specific upper and lower adfs. %This question is addressed by \refmt{Main}, which also solve the following six problems for $Q$-Cantor series expansions when $Q$ is infinite in limit.
This will allow us to attack a wide range of problems.

%We recall a fundamental theorem on normal numbers.
%\begin{thrm}
%A real number $x$ is normal in base $b$ if and only if the sequence $\{b^nx\}$ is uniformly distributed mod $1$.
%\end{thrm}

%As a consequence of our main theorem \refmt{Main} we will resolve the following problems for $Q$-Cantor series when $Q$ is infinite in limit.

\begin{definition}
A set of  functions $\{ f_{m,r} \}_{m \in \mathbb{N}, 0\leq r < m}$ is called a \textbf{linear family} if for all $m$, $r$, and $d$
\begin{equation*}
f_{m,r} = \frac{1}{d} \sum_{i=0}^{d-1} f_{md, mi+r}.
\end{equation*}
\end{definition}
%\textbf{linear family of asymptotic distribution functions}
The set of functions $\{f_{m,r}\}$ where $f_{m,r}(x) = x$ for all $m$ and $r$ gives an example of a linear family of adfs.
A non-trivial example is given in the proof of \reft{NormNotAP} in \refs{OtherProofs}.

For $p$, $q \in \mathbb{Z}[X]$, set
\begin{equation*}%\labeq{dpqsn}
\de(p,q,s,n) := \Fpqstart{p}{q}{s}{n}.
\end{equation*}
%This is the relative density of the images of one polynomial in another.
Define a relation $\lesssim$ among polynomials $p$, $q \in \mathbb{Z}[X]$ so
$$
q \lesssim p \hbox{ if } \varlimsup_{n-s \to \infty} \de(p,q,n,s) > 0.
$$
If  $p \lesssim q$ and $q \lesssim p$, then we write $p \approx q$.  It is easy to verify that $\lesssim$ is a preorder but not a partial order. Similarly, we can show that $\approx$ is an equivalence relation.

\begin{definition}
A set of polynomials $P \subseteq \mathbb{Z}[X]$ % indexed by $\mathbb{N}$
 is \textbf{saturated} if for any $f \in \mathbb{Z}[X]$ there exists a polynomial $p \in P$ and a linear polynomial $\mu \in \mathbb{Z}[X]$ %and $k \in \mathbb{N}$ 
such that $f = p \circ \mu$. The set $P$ is \textbf{sparsely intersecting} if for each $i$ and $j$ we have $p_i \lnsim p_j$ or $p_j \lnsim p_i$.
\end{definition}

We will prove the following lemma in \refs{polyisect}.
\begin{lem}\labl{iscomplete}
There is a saturated sparsely intersecting set of polynomials.
\end{lem}

A real number $x$ is \textbf{computable} if there exists $b \in \mathbb{N}$ with $b\geq 2$ and a total recursive function $f: \mathbb{N} \to \mathbb{N}$ that calculates the digits of $x$ in base $b$. A sequence of real numbers $\{x_n\}$ is \textbf{computable} if there exists a total recursive function $f: \mathbb{N}^2 \to \mathbb{Z}$ such that for all $m,n$ we have that $\frac{f(m,n)-1}{m} < x_n< \frac{f(m,n)+1}{m}$. 
A sequence of functions $\{f_n\}$ from a metric space $X$ to $\mathbb{R}$ is \textbf{uniformly computable} if the double sequence $\{f_n(x_m)\}$ is computable for any computable sequence $\{x_m\}$ and if there is a recursive function $\gamma(n,k)$ such that for all $n,k$ and $x,y \in X$, we have that $d(x,y)\leq \frac{1}{2^{\gamma(n,k)}}$ implies $|f_n(x)-f_n(y)|\leq \frac{1}{2^k}$. A function $f$ is \textbf{uniformly computable} if the sequence $\{f,f,f,\cdots \}$ is uniformly computable. A sequence $\{x_n\}$ is \textbf{uniformly computable} if there is a uniformly computable function $f: \mathbb{N}\to \mathbb{R}$ such that $f(n) = x_n$ \cite{TakaMori}.
%These notions of computability will be used in \reft{explicit} to make the construction explicit.

\begin{definition}
A basic sequence $Q = \{q_n\}$ 
is a \textbf{computably growing basic sequence} if it is infinite in limit and
the sequence $\{ \inf\{i : \forall j \geq i \ (q_j \geq n) \} \}_{n=1}^\infty$ is computable.
\end{definition}

\begin{definition}
A linear family of adfs $\{ f_{m,r} \}_{m \in \mathbb{N}, 0\leq r < m}$  is an \textbf{explicit linear family of adfs} if for each $m,r \in \mathbb{N}$ with $0\leq r < m$ the following hold.
\begin{enumerate}
\item The real numbers $f_{m,r}(q)$ and $\inf \br{x\in [0,1] : f_{m,r}(x) = q}$ are computable for every rational $q \in [0,1]$.

\item If $f_{m,r}$ is discontinuous at $t$, then $t$ is a computable real number and $f_{m,r}(t)$ is a computable real number.

\item The function $f_{m,r}$ is either continuous, has only finitely many discontinuities, or the set of its discontinuities may be written in the form $\br{t_n : n \in \mathbb{N}}$, where $\br{t_n}$ is a uniformly computable sequence.
\end{enumerate}
\end{definition}

\begin{definition}
A sparsely intersecting set of polynomials $P = \{ p_i \}$ is an \textbf{explicit sparsely intersecting set of polynomials} if for all
 %natural numbers $i$ and $j$ 
$p, q \in P$ there exists a computable sequence $\{N(m)\}$ such that if $n-s>N(m)$, then $d_{p,q,s,n}<\frac{1}{m}$ or $d_{q,p,s,n} < \frac{1}{m}$.
\end{definition}

%\begin{definition}
%An \textbf{explicit tuple} is a tuple $(Q,P,F)$ of a computably growing basic sequence, a explicit sparsely intersecting set of polynomials containing $p(n) = n$, and explicit linear families of upper and lower adfs for each $p \in P$.
%\end{definition}

%Whenever there is no chance of confusion, we will write $f_{p_i, m, r}$ as $f_{i,m,r}$.
Given a basic sequence $Q$, a set of sparseley intersecting polynomials $P$, and a set of linear families 
\begin{equation}\labeq{Fadf}
F=\br{\br{ \overline{f}_{p,m,r} }_{m \in \mathbb{N}, 0\leq r < m}}_{p \in P} \cup \br{\br{ \underline{f}_{p,m,r} }_{m \in \mathbb{N}, 0\leq r < m}}_{p \in P}
\end{equation}
 of  adfs, define
\begin{align*}
\Phi_{Q,P,F} = \Big\{ x \in [0,1) : &\OQfmrx{p} \text{ has upper and lower adfs } \\
&\overline{f}_{p,m,r}, \underline{f}_{p,m,r}\in F, \forall p \in P, m \in \mathbb{N}, 0\leq r < m\Big\}.
\end{align*}

We may now state the main theorem of this paper.

\begin{mainthrm}\labmt{Main}
If $Q$ is infinite in limit, $P$ is a set of sparsely intersecting polynomials, and $F$ is a set of linear families of upper and lower adfs given by \refeq{Fadf}, then
%$F = \left\{\overline{f}_{p_i,m,r}, \underline{f}_{p_i,m,r}\ : m\in \mathbb{N} \text{ and } 0\leq r < m \right\}$ for each $p_i \in P$
\begin{equation*}
\dimh{\Phi_{Q,P,F}} = 1.
\end{equation*}
Furthermore, if $Q$ is computable and computably growing, $P$ is explicit, and $F$ is explicit, then there is a subset $\Phi'_{Q,P,F}$ of $\Phi_{Q,P,F}$ such that the following hold.
\begin{enumerate}
\item The set $\Phi'_{Q,P,F}$ has full Hausdorff dimension.
\item There exist computable sequences $\{\alpha(n)\}$ and $\{\beta(n)\}$ such that $$ \Phi'_{P,Q,F} = \{ x \in [0,1) : \alpha(n) \leq E_n(x) \leq \beta(n)\}.$$
\end{enumerate}
The set $\Phi'_{Q,P,F}$ and computable sequences $\{\alpha(n)\}$ and $\{\beta(n)\}$ are constructed in \refs{construction}.
\end{mainthrm}
\refmt{Main} is proven in \refs{construction}.

\subsection{Application I: Equivalent definitions of normality}

We recall the modern definition of a normal number.

\begin{defin}\labd{normal}
A real number $x$ is {\bf normal of order $k$ in base $b$} if all blocks of digits of length $k$ in base $b$ occur with relative frequency $b^{-k}$ in the $b$-ary expansion of $x$. Moreover, $x$ is {\bf simply normal in base $b$} if it is normal of order $1$ in base $b$ and $x$ is {\bf normal in base $b$} if it is normal of order $k$ in base $b$ for all natural numbers $k$.
\end{defin}

It is well known that \'{E}. Borel \cite{BorelNormal} was the first mathematician to study normal numbers.  In 1909 he gave the following definition.

\begin{defin}[\'{E}. Borel]\labd{normalBorel}
A real number $x$ is {\bf normal in base $b$} if each of the numbers $x, bx, b^2x,\cdots$ is simply normal (in the sense of \refd{normal}), in each of the bases $b, b^2, b^3, \cdots$.
\end{defin}

E. Borel proved that Lebesgue almost every real number is normal, in the sense of \refd{normalBorel}, in all bases.  In 1940, S. S. Pillai \cite{Pillai2} simplified \refd{normalBorel} by proving that
\begin{thrm}[S. S. Pillai]\labt{Pillai}
For $b \geq 2$, a real number $x$ is normal in base $b$ if and only if it is simply normal in each of the bases $b, b^2, b^3, \cdots$.
\end{thrm}
\reft{Pillai} was improved in 1951 by I. Niven and H. S. Zuckerman \cite{NivenZuckerman} who proved
\begin{thrm}[I. Niven and H. S. Zuckerman]\labt{NivenZuckerman}
\refd{normal} and \refd{normalBorel} are equivalent.
\end{thrm}

A simpler proof of \reft{Pillai} was given by J. E. Maxfield in \cite{MaxfieldPillai}.  J. W. S. Cassels gave a shorter proof of \reft{NivenZuckerman} in \cite{CasselsNZ}.
 It should be noted that both of these results require some work to establish, but were assumed without proof by several authors.  For example, M. W. Sierpinski assumed \reft{Pillai} in \cite{Sierpinski} without proof.  Moreover, D. G. Champernowne \cite{Champernowne}, A. H. Copeland and P. Erd\H os \cite{CopeErd}, and other authors took \refd{normal} as the definition of a normal number before it was proven that \refd{normal} and \refd{normalBorel} are equivalent.  More information can be found in Chapter 4 of the book of Y. Bugeaud \cite{BugeaudBook}.

%Less well known is another interpretation of ``AP normality'' is the following equivalent definition of normality.  WORDING
The following theorem was proven by H. Furstenberg in his seminal paper ``Disjointness in Ergodic Theory, Minimal Sets, and a Problem in Diophantine Approximation'' \cite{FurstenbergDisjoint} on page 23 as an application of disjointness to stochastic sequences.
\begin{thrm}[H. Furstenberg]\labt{basebnormalii}
Suppose that $x=d_0.d_1d_2\cdots$ is the $b$-ary expansion of $x$.  Then $x$ is normal in base $b$ if and only if for all natural numbers $m$ and $r$ the real number $0.d_{r}d_{m+r}d_{2m+r}d_{3m+r}\cdots$ is normal in base $b$.
\end{thrm}

%H. Furstenberg mistakenly claimed (see appendix) that this provided an alternate proof of \reft{NivenZuckerman}, but actually proved that an entirely different definition of normality is equivalent to \refd{normal}.  
We will say that $x$ is \textbf{AP normal of type I in base $b$} if $x$ satisfies \refd{normalBorel} and \textbf{AP normal of type II in base $b$} if $x$ satisfies the notion introduced in \reft{basebnormalii}.  
It is interesting to note that although Furstenberg did not provide an alternate
proof of \reft{NivenZuckerman}, he did define an equivalent notion of normality, that is AP normality of type II. See appendix for more details.
Thus, for numbers expressed in base $b$
\begin{equation}\labeq{APequiv}
\hbox{normality }\Leftrightarrow \hbox{ AP normality of type I } \Leftrightarrow \hbox{ AP normality of type II}.
\end{equation}

The authors feel that the equivalence of \refd{normal} and \refd{normalBorel} and other similar ones is a far more delicate topic than is typically assumed.  
The core of E. Borel's definition is that a number is normal in base $b$ if blocks of digits occur with the desired relative frequency along all infinite arithmetic progressions.
%We will extend \refd{normalBorel} to the $Q$-Cantor series expansions and call it {\bf AP $Q$-normality of type I}.  P. Laffer has already studied a similar definition in \cite{Laffer}, but made no comparison with the definition of $Q$-normality that naturally arises as an extension of \refd{normal}.  We will prove that the analogous extension of \reft{NivenZuckerman} to the $Q$-Cantor series expansions will no longer hold.  
%We will refer to the natural extension of the concept introduced in \reft{basebnormalii} to the $Q$-Cantor series expansions as {\bf AP $Q$-normality of type II}.  
%We will show that for $Q$-Cantor series expansions that while $Q$-normality is implied by AP $Q$-normality of type I, the bigger picture of the relationship between these notions  is far more complex than for the $b$-ary expansions.
We say that a real number $x$ is \textbf{$Q$-distribution normal} if $\OQx$ is u.d. mod 1.  A real number $x$ is \textbf{AP $Q$-distribution normal of type I} if for all $m \in \mathbb{N}$ and $0 \leq r < m$ we have that $\OQmrx$ is u.d. mod 1. If $x = E_0.E_1 E_2 \cdots \wrtQ$, then we say that $x$ is \textbf{AP $Q$-distribution normal of type II} if the real number $0.E_r E_{m+r} E_{2m+r} \cdots$ is $\{ q_{m(n-1)+r}\}_{n=1}^\infty$-distribution normal for all $m \in \mathbb{N}$ and $0 \leq r < m$. 
%There are non-equivalent ways of defining normality based on the distribution of digits
%We will show that if $Q$ is infinite in limit, then AP $Q$-distribution normality of types I and II are equivalent. 
We say that $x$ is \textbf{AP $Q$-distribution abnormal} if $\OQmrx$ is not u.d. mod 1 for any $m>1$.

%SAY HOW THIS RELATES TO DIFFERENT TYPES OF NORMAL NUMBERS
%$$
%\OQfmrx{f}=\br{T_{Q,f(mn+r)}(x)}_{n=0}^\infty
%$$
%polynomial sequences and normality

%some titles need changing

We will prove in \refs{OtherProofs} that $Q$-distribution normality is not equivalent to AP $Q$-distribution normality in a particularly strong way. 
The following theorem describes exactly how much \refeq{APequiv} may be extended to $Q$-Cantor series expansions when $Q$ is infinite in limit.

\begin{thrm}\labt{NormNotAP}
Let $Q$ be a basic sequence that is infinite in limit. Then 
\begin{enumerate}
\item AP $Q$-distribution normality of type I is equivalent to AP $Q$-distribution normality of type II.
\item The set of real numbers which are $Q$-distribution normal and AP $Q$-distribution abnormal is a meagre set with zero measure and full Hausdorff dimension.
\end{enumerate}
\end{thrm}

Furthermore, if $Q$ is computable and computably growing, then the proof of \reft{NormNotAP} provides a computable example of a real number that is $Q$-distribution normal and AP $Q$-distribution abnormal.  See \refs{constructionnormal} for further discussion.  However, \refmt{Main} is far stronger since it allows us to specify upper and lower adfs along polynomially indexed subsequences of $\OQx$. \reft{NormNotAP} only requires knowledge of $\OQx$ along infinite arithmetic progressions.

We note that far less is known if we  extend \refeq{APequiv} to analogous definitions involving digits.  The problem is discussed in \cite{ppq2} and partial results are given.  One substantial difference is that the analogous generalizations of the  definitions of AP normality of types I and II are no longer equivalent.  However, these definitions are technical, so we choose not to state any of these results here.

%We will show that distribution normality is strongly not equivalent to AP distribution normality. We can take different notion of normality involving the digits rather than the distribution of the sequence $\OQx$. This type of normality has different properties, and so AP normality of types I and II are not equivalent. This is studied in more depth by B. Li and the second author in \cite{ppq2}.

%One of the main motivations of this paper is to prove that the set of real numbers that are distribution normal but AP-abnormal has full Hausdorff dimension, which is  to the results by H. Furstenburg, I. Niven, and H.S. Zuckerman for b-ary expansions.
%We extend these results even further by considering distribution normality along polynomial sequences. We will construct a set of real numbers $x$ of full Hausdorff dimension such that $\OQfmrx{p}$ has the desired distribution for each $x$ and for each polynomial $p \in \mathbb{Z}[X]$ with positive leading coefficient. 

\subsection{Application II: Computing the Hausdorff dimension of Besicovitch--Eggleston normal sets}
% sets of real numbers whose digits have specified frequencies}

 The following well known result was proved for $b=2$ by A. S. Besicovitch in \cite{BesicovitchEggleston} and for all other $b$ by H. Eggleston in \cite{Eggleston}.

\begin{thrm}[H. Eggleston]\labt{Eggleston}
Let $b \in \mathbb{N}_2$ and $\vec{p}=(p_0,p_1,\cdots,p_{b-1})$ be a probability vector.  Then the Hausdorff dimension of the set of all real numbers $x$ where the digit $i$ occurs in the $b$-ary expansion of $x$ with relative frequency $p_i$ for all $i=0,1,\cdots,b-1$ is equal to
$$
\frac {-\sum_{i=1}^{b-1} p_i\log p_i} {\log b}.
$$
\end{thrm}
%We refer to the sets described in \reft{Eggleston} as {\it Besicovitch-Eggleston sets}.  
There have been numerous improvements of \reft{Eggleston}.  Moreover, \reft{Eggleston} has been extended to certain classes of Cantor series expansions.  Early work was done by J. Peyri\`{e}re in \cite{Peyriere} and Y. Kifer in \cite{Kifer}.  We mention a similar result  proven by Y. Xiong in \cite{Xiong}.  For a given basic sequence $Q$, let $N_n^Q(B,x)$ denote the number of times a block $B$ occurs starting at a position no greater than $n$ in the $Q$-Cantor series expansion of $x$.

\begin{thrm}[Y. Xiong]\labt{Xiong}
Suppose that $Q$ is infinite in limit and that $\vec{p}=(p_n)$ is an infinite probability vector.    For $m>0$, let
\begin{align*}
B_m(\vec{p})&=\left\{ x \in [0,1) : \lim_{n \to \infty} \frac {N_n^Q(k,x)} {n}=p_k,\hbox{ for }0 \leq k \leq m\right\};\\
B(\vec{p})&=\left\{ x \in [0,1) : \lim_{n \to \infty} \frac {N_n^Q(k,x)} {n}=p_k,\hbox{ for } k \geq 0\right\}.
\end{align*}
Then the following hold.
\begin{enumerate}
\item If $\lim_{n \to \infty} \frac {\log q_n} {\sum_{j=1}^n \log q_j}=0$, then
$$
\dimh{B_m(\vec{p})}=\sup_{t \in T_m} \liminf_{n \to \infty} \frac {\sum_{j=1}^n t_j \log q_j} {\sum_{j=1}^n \log q_j},
$$
where 
$$T_m=\left\{t \in \{0,1\}^{\mathbb{N}}: \frac {\sum_{j=1}^n t_j} {n}=\sum_{j=m}^\infty p_j\right\}.
$$
\item If $Q$ is increasing and the sequence $\left\{\frac {\log q_n} {\sum_{j=1}^n \log q_j}\right\}$ is bounded,  then $$\dimh{B(\vec{p})}=~0.$$
\end{enumerate}
\end{thrm}

C. M. Colebrook \cite{Colebrook} proved a similar result to \reft{Eggleston} about the Hausdorff dimension of the set of real numbers $x$ where the sequence $\Obx$ has a given adf. A special case of \refmt{Main}  extends C. M. Colebrook's result to a  large class of Cantor series expansions.

\begin{thrm}\labt{extendcolebrook}
Suppose that $Q$ is infinite in limit and that $f$ is an adf.  Then the set of real numbers $x$ such that $\OQx$ has adf of $f$ has full Hausdorff dimension.
\end{thrm}

We note that the sets considered in \reft{Xiong} have much smaller Hausdorff dimension than those considered in \reft{extendcolebrook}.  This is in sharp contrast to the case of the $b$-ary expansions.

\subsection{Application III: Analyzing the Hausdorff dimension of sets of numbers which may not have digital frequencies}

It is difficult in general to analyze sets of real numbers some of whose digital frequencies  may not exist. This is discussed in L. Olsen's paper \cite{OlsenDigits}.  For an integer $b \geq 2$, let $N_n^b(B,x)$ denote the number of times a block $B$ occurs starting at a position no greater than $n$ in the $b$-ary expansion of $x$ and set 
$$
\Pi_b(B,x):=\lim_{n \to \infty} \frac {N_n^b(B,x)} {n}.
$$
\begin{thrm}[L. Olsen]\labt{Olsen}
For $k>0$ and $b \in \mathbb{N}_2$, put
\begin{align*}
&\mathscr{A}=\br{x \in [0,1] : \Pi_b(1,x)=k\Pi_b(0,x)};\\
&\mathscr{B}=\br{x \in [0,1] : \Pi_2(000,x)=\Pi_2(11,x)};\\
&\mathscr{C}=\br{x \in [0,1] : \Pi_3(1,x)^2=\Pi_3(0,x)}.
\end{align*}
Then
\begin{align*}
&\dimh{\mathscr{A}}=\frac {\log \pr{(k+1)/k^{k/(k+1)}+b-2}}{\log b};\\
&\dimh{\mathscr{B}}=0.688593\cdots;\\
&\dimh{\mathscr{C}}=0.9572506922\cdots.
\end{align*}
\end{thrm}
We note that L. Olsen provides a precise description of the Hausdorff dimension of $\mathscr{B}$ and $\mathscr{C}$.  We intentionally omit these descriptions here and only provide an approximation as they are relatively complex. We will prove the following extension of \reft{Olsen}  in \refs{OtherProofs}.
\begin{thrm}\labt{OlsenGen}
Suppose that $I$ and $J$ are disjoint non-degenerate intervals %with non-zero measure 
and that $f$ and $g$ are real valued functions on $[0,1]$ which satisfy the following conditions:
\begin{itemize}
\item There is an $x_0$ in $[0,1/2]$ such that $f(x_0) = g(x_0)$;
\item We have that  $\lambda(I) + \lambda(J) < 1$.

\end{itemize}
If $Q$ is infinite in limit, then 
$$
\dimh {\br{ x \in [0,1) : f\pr{\lim_{ n \to \infty} \frac{A_n ( I, \OQx)}{n}  }= g\pr{\lim_{ n \to \infty} \frac{A_n(J, \OQx)}{n}} } } = 1.
$$
\end{thrm}
%SAY OUR TECHNIQUES ARE DIFFERENT
%REMARK HOW THIS IS STATED WITH OLSEN LANGUAGE

%PUT THING IN PAPER OF OLSON ABOUT SINES AND NON LINEAR STUFF (SEE REVIEW ON MATHSCINET) Hausdorff and packing dimensions of non-normal tuples of numbers: non-linearity and divergence points.

S. Albeverio, M. Pratsiovytyi, and G. Torbin proved the following theorem in \cite{topfracnotnorm} 

\begin{thrm}[S. Albeverio, M. Pratsiovytyi, and G. Torbin]\labt{noFrequencyAPT}
The set of real numbers whose frequencies of digits in base $b$ do not exist has zero measure and full Hausdorff dimension.
\end{thrm}

We extend their result to the following theorem.

\begin{thrm}\labt{noFrequency}
If $Q$ is infinite in limit, then the set of real numbers $x$ such that $\OQx$ has no adf has zero measure and full Hausdorff dimension.
\end{thrm}

\reft{noFrequency} is proven in \refs{OtherProofs}.

\subsection{Application IV Constructing examples of normal numbers}\labs{constructionnormal}

The most well known construction of a normal number in base $10$ is due to Champernowne. The number
\begin{equation*}
0.1 \ 2 \ 3 \ 4 \ 5 \ 6 \ 7 \ 8 \ 9 \ 10 \ 11 \ 12 \cdots,
\end{equation*}
formed by concatenating the digits of every natural number written in increasing order in base 10, is normal in base 10. The number formed by concatenating the digits of the natural numbers in base $b$ in order is normal in base $b$.
A. H. Copeland and P. Erd\"{o}s in \cite{CopeErd} showed that the number 
$$
0.2\ 3\ 5\ 7\ 11\ 13\ 17\ 19\ 23\ 29\cdots,
$$
formed by concatenating the digits of all prime numbers is normal in base $b$. H. Davenport and P. Erd\"{o}s in \cite{DavenErd} showed that the number formed by concatenating the value of a positive integer valued polynomial at each natural number yields a normal number in base $b$. Many similar and more sophisticated results have been proven since then. For example,  J. Vandehey \cite{VandeheyAdditive} and M. Madritsch and R. Tichy \cite{MaTichy} have given similar constructions.
A more extensive list of results can be found in Y. Bugeaud's book \cite{BugeaudBook}. %computability, very clear when constructions give examples of computable functions, Theorem 2.9

A real number is \textbf{absolutely normal} if it is normal in base $b$ for all $b \in \mathbb{N}_2$.
M. W. Sierpi\'{n}ski gave an example of an absolutely normal number that is not computable in \cite{Sierpinski}.
The authors feel that examples such as M. W. Sierpi\'{n}ski's are not fully explicit since they are not computable real numbers, unlike Champernowne's number. 
%V. Becher and S. Figueira construct a computable real number that is absolutely normal in \cite{BecherFigueira}. 
A. M. Turing gave the first example of a computable absolutely normal number in an unpublished manuscript.  This paper may be found in his collected works \cite{Turing}. The $n$'th digit of A. M. Turing's number may be computer with an algorithm that is doubly exponential in $n$. V. Becher, P. A. Heiber, and T. A. Slaman constructed an absolutely normal number in \cite{BecherHeiberSlaman} whose digits may be computed in polynomial time.
See \cite{BecherFigueiraPicchi} by V. Becher, S. Figueira, and R. Picchi for further discussion.

We will use \refmt{Main} to construct a computable $Q$-distribution normal number when $Q$ is computable and computably growing.

\begin{thrm}\labt{computableElement1}
Suppose that $Q$ is computable and computably growing, $P=\{X\}$, and $\overline{f}_{X,1,0}(x)=\underline{f}_{X,1,0}(x)=x$.  If $\{\alpha(n)\}$ is the sequence given in \refmt{Main}, then the real number $\sum_{n=1}^\infty \frac{\alpha(n)}{q_1 \cdots q_n}$ is computable and $Q$-distribution normal.
%$x \in \Phi'_{Q,P,F}$.
\end{thrm}

P. Lafer \cite{Lafer} asked for a construction of a $Q$-distribution normal number for an arbitrary basic sequence $Q$.  \reft{computableElement1} provides a partial answer to this question.  
However, we remark that there are basic sequences $Q$ such that no $Q$-distribution normal number is computable.  A. A. Beros and K. A. Beros showed in \cite{BerosBerosComputable} that there exists a limit computable basic sequence $Q$ such that no computable real number is $Q$-distribution normal.
Thus, it is impossible to answer P. Laffer's question for an arbitrary basic sequence.  It remains open how much one needs to assume about a basic sequence $Q$ in order to guarantee existence of computable $Q$-distribution normal numbers and how to construct one of these numbers.

\subsection{Application V: Constructing  examples of real numbers with different digital frequencies}

There is substantial literature in pertaining to the explicit construction of numbers with different digital frequencies. The notes of section 1.8 in \cite{KuN} provide a good list of papers on the subject. We solve an analogous problem for Cantor series expansions with $Q$ infinite in limit: constructing a computable real number $x$ so that $\OQx$ has a given adf $\phi$. This follows immediately from \refmt{Main} and is a more general version of \reft{computableElement1}.

\begin{thrm}\labt{computableElement2}
Suppose that the singleton $\{\phi \}$ is an explicit set of adfs, $Q$ is computable and computably growing, $P=\{X\}$, and $\overline{f}_{\operatorname{X},1,0}=\underline{f}_{\operatorname{X},1,0}=\phi$.  If $\{\alpha(n)\}$ is the sequence given in \refmt{Main}, then the real number $\xi=\sum_{n=1}^\infty \frac{\alpha(n)}{q_1 \cdots q_n}$ is computable and $\OQ{\xi}$ has adf $\phi$.
\end{thrm}

%first, then eggleston, then frequencies don't exist, the constructing normal numbers, then cosntructing diff digit freq, sharpening known theorems.

%explaining how we get those sums

\subsection{Application VI: Sharpening known theorems}

We mention four results from other papers that will follow as corollaries of our main theorem.  In fact the immediate corollaries will be stronger than the results stated in this section.  The following theorem was proven by J. Peyri\`{e}re in \cite{Peyriere}.

\begin{thrm}[J. Peyri\`{e}re]\labt{Pey}
Suppose that $Q$ is infinite in limit.  Then for all $\ell \in (0,1)$
$$
\dimh{\left\{x=0.E_1E_2\cdots \hbox{ w.r.t. }Q: \lim_{n \to \infty} \frac {1} {n}  \sum_{j=1}^n \frac {E_j} {q_j}=\ell\right\}  }=1.
$$
\end{thrm}

For any sequence $X=\{x_n\}$ of real numbers, let $\mathbb{A}(X)$ denote the set of accumulation points of $X$. 
Given a set $D \subseteq [0,1]$, let
$$
\mathbb{E}_D(Q):=\left\{x=0.E_1E_2\cdots \hbox{ w.r.t. }Q: \mathbb{A}(\OQxp)=D      \right\}.
$$
The following results are proven by Y. Wang, Z. Wen, and L. Xi in \cite{WangWenXi}.
\begin{thrm}[Y. Wang, Z. Wen, and L. Xi]\labt{WangWenXi}
If $Q$ is infinite in limit, then $\dimh{\mathbb{E}_D(Q)}=1$ for every closed set $D$.
\end{thrm} 

Let 
$$
\mathbb{E}_{D,m,r}(Q):=\left\{x=0.E_1E_2\cdots \hbox{ w.r.t. }Q: \mathbb{A}(\OQmrxp)=D_{m,r}\right\}.
$$
Note that \reft{Pey} follows as a direct consequence of \reft{WangWenXi} by considering the closed set $D = \{\ell \}$ for $0\leq \ell \leq 1$.
We will prove the following theorem as a corollary of \refmt{Main} in \refs{arith}.
\begin{thrm}\labt{strongerWangWenXi}
Suppose that  $k \in \mathbb{N}$ and closed sets $D_{m,r}$ are contained in $[0,1]$ for $1\leq m \leq k$, and $0\leq r < m$. If $\bigcap_{m=1}^k \bigcap_{r=0}^{m-1} D_{m,r} \neq \emptyset$, then 
$$
\dimh{\bigcap_{m=1}^k \bigcap_{r=0}^{m-1} \mathbb{E}_{D_{m,r},m,r}(Q)} =1.
$$
%the set of real numbers $x$ so that the accumulation points of $\OQmrxp =D_{m,r}$ has Hausdorff dimension 1.
\end{thrm}
Note that \reft{WangWenXi}  follows immediately from \reft{strongerWangWenXi} by setting $k=1$.  Similarly, we may obtain \reft{Pey} by setting $k=1$ and $D_{1,1}=\{\ell\}$.

In  \cite{MyersonPollington} G. Myerson and A. D. Pollington proved the following result.

\begin{thrm}[G. Myerson and A. D. Pollington]\labt{mp}
There exists a uniformly distributed sequence $X=\{x_n\}$ such that none of the sequences $\{x_{kn+j}\}$ are u.d. mod 1 for any $k \geq 2$. 
\end{thrm}
Moreover, G. Myerson and A. D. Pollington provide an example of a sequence with the property described in \reft{mp}.\\

Let $Q=\{q_n\}$ be given by $q_n=n+1$.  By \reft{NormNotAP} there exists a computable real number $\xi$ that is $Q$-distribution normal and AP $Q$-distribution abnormal. Setting $x_n=T_{Q,n-1}(\xi)$, we see that \reft{mp} follows immediately from \reft{NormNotAP}. Moreover, the sequence $\{x_n\}$ is computable. Thus, \refmt{Main} provides a strong generalization of \reft{mp}.  

Lastly, given any adf $f$, J. G. van der Corput constructed a sequence with adf $f$ in \cite{CorputDistFunc}.  The proof of \refmt{Main} gives a far more general construction.

%\begin{cor}\labc{WangWenXi}
%Given $0 \leq \delta \leq 1$, let 
%$$
%\mathbb{E}_\delta(Q)=\mathbb{E}_{\{\delta\}}(Q)=\left\{ x=0.E_1E_2\cdots \hbox{ w.r.t. }Q: \lim_{n \to \infty} \frac {E_n} {q_n}=\delta     \right\}.
%$$
%If $Q$ is infinite in limit, then $\dimh{\mathbb{E}_\delta(Q)}=1$.
%\end{cor}

%----------------------------------------------------------------------------------------------------------------------------------
\section{Main theorem and construction}

%\subsection{Asymptotic distribution functions}% along arithmetic sequences}
%Define A_n of interval, sequence of real numbers has adf

In order to prove \refmt{Main}, we will need to construct the set $\Phi'_{Q,P,F}$.

\subsection{Intersection of Polynomial Sequences}\labs{polyisect}
To construct a sequence $\OQx$ with desired distributional properties along polynomially indexed subsequences we will need to estimate $\de(p,q,s,n)$.

We say an equation $F(x,y)=0$ with $F(x,y) \in \mathbb{Z}[X,Y]$ has infinitely many solutions with bounded denominator if there is some $\Delta \in \mathbb{Z}$ such that there are infinitely many $(x,y) \in \mathbb{Q} \times \mathbb{Q}$ with $\Delta x, \Delta y \in \mathbb{Z}$ and $F(x,y)=0$. If there are infinitely many integer solutions, there are infinitely many solutions of bounded denominator.

We will need the following theorem by Y. Bilu and R. Tichy \cite{BiluTichy}.

\begin{thrm}[Y. Bilu and R. Tichy]\labt{BiluTichy}
Let $p,q \in \mathbb{Z}[X]$. The equation $p(x)=q(y)$ has infinitely many solutions of bounded denominator if and only if there exist a polynomial $\phi \in \mathbb{Q}[X]$, linear polynomials $u, v \in \mathbb{Q}[X]$ and polynomials $f,g \in \mathbb{Z}[X]$ where $(f,g)$ is a standard pair such that $\phi \circ p \circ f \circ u = \phi \circ q \circ g \circ v$. The standard pairs are
\begin{enumerate}
	\item $(x^m, ax^r s(x)^m)$ with $0\leq r < m, (r,m)=1$ and $r+deg(p)>0$;
	\item $(x^2, (ax^2+b) s(x)^2)$;
	\item $(D_m(x, a^n), D_n(x, a^m))$, where $D_m(x,a)$ is the $m$-th Dickson polynomial and $(m,n)=1$;
	\item $(a^{-m/2} D_m(x,a), -b^{-n/2} D_n(x,b) )$ with $(m,n)=2$;
	\item $( (ax^2-1)^3, 3x^4-4x^3)$.
	
	\end{enumerate} where $s(x) \in \mathbb{Q}[X]$ is non-zero (possibly constant polynomial, and $a,b \in \mathbb{Q}$ are non-zero.
\end{thrm}

With this theorem we can prove the following.

\begin{lem}\labl{SparseIntersect}
If $p, q \in \mathbb{Z}[X]$, then $p \approx q$ if and only if there exist linear polynomials $\mu$, $\lambda$ $\in \mathbb{Z}[X]$ such that $p \circ \mu = q \circ \lambda$.
\end{lem}

\begin{proof}

For the forward direction, suppose that two linear polynomials $\mu(n) = mn+r$ and $\lambda(n) = m'n + r'$ exist so that $p \circ \mu = q \circ \lambda$. Then 
$$
\varlimsup_{s-n \to \infty} \de(p,q,s,n) \geq \frac{1}{m'}>0.
$$
 Similarly, $\varlimsup_{s-n \to \infty} \de(q,p,s,n) \geq \frac{1}{m}>0$.

For the reverse direction we look at two cases: when $\deg(p)=\deg(q)$ and when $\deg(p)\neq \deg(q)$.
Suppose first that $\deg(p) = k > \deg(q) = l$. Write 
\begin{align*}
p(n) &= a_k n^k + \cdots + a_1 n + a_0;\\
q(n) &= b_l n^l + \cdots + b_1 n + b_0.
\end{align*} Then 
$$
\varlimsup_{s-n \to \infty} \frac{\# p(\mathbb{N})\bigcap \{s, s+1, \cdots, n\}}{p^{-1}(n)-p^{-1}(s)} = 1.
$$ 
But 
$$
\varlimsup_{s-n \to \infty} \frac{p^{-1}(n)-p^{-1}(s)}{\frac{1}{a_k} n^{\frac{1}{k}}-{a_k} s^{\frac{1}{k}}} = 1,
$$
so 
$$
\varlimsup_{s-n \to \infty} \frac{a_k \# p(\mathbb{N})\bigcap \{s, s+1, \cdots, n\}}{n^{\frac{1}{k}}-s^{\frac{1}{k}}} = 1.
$$ 
Similarly, 
$$
\varlimsup_{s-n \to \infty} \frac{b_l \# q(\mathbb{N})\bigcap \{s, s+1, \cdots, n\}}{n^{\frac{1}{l}}-s^{\frac{1}{l}}} = 1.
$$ 
Since 
$$
\# (p(\mathbb{N}) \bigcap q(\mathbb{N}) \bigcap \{ 1, 2, \cdots n\}) \leq \# (p(\mathbb{N})\bigcap \{1, 2, \cdots, n\}),
$$ 
we have that 
$$
\varlimsup_{n-s \to \infty} \de(p,q,s,n) \leq \varlimsup_{n-s \to \infty} \frac{b_l n^\frac{1}{k}-b_l s^\frac{1}{k}}{a_k n^\frac{1}{l}-a_k s^\frac{1}{l}} = 0.
$$
 But $p$ and $q$ have different degrees, so they cannot be equal when composed with linear polynomials. Thus the theorem holds for this case.

Now suppose that $\deg(p) = \deg(q) = k$. In order to have that $p \approx q$, there must be infinitely many integer solutions to the equation $p(x)=q(y)$. By \reft{BiluTichy}, we must have that there exists a polynomial $\phi \in \mathbb{Q}[X]$, linear polynomials $\mu, \lambda \in \mathbb{Q}[X]$, and a standard pair of polynomials $f,g \in \mathbb{Z}[X]$.

Since $\deg(p)=\deg(q)$, we must have that $\deg(f) = \deg(g)$. But there are only a few cases where a standard pair of polynomials can have the same degree. For standard pairs of the first and third type, we must have that $f$ and $g$ are linear. For standard pairs of the second and fourth type, we must have that $f$ and $g$ are of the form $an^2+b$ with $a,b \in \mathbb{Q}$. Standard pairs of the fifth kind cannot have equal degrees. If $f$ and $g$ are linear polynomials, then the proof of the previous direction suffices. 

We only need to prove the claim when $f$ and $g$ are quadratic with zero linear term, or equivalently when $f(n) = n^2$ and $g(n)=an^2+b$. In this case, note that the equation $x^2 = dy^2+e$ is equivalent to $ax^2+by^2=c$ with integers $a$, $b$, and $c$. The solutions of this equation are distributed accoring to a number of different relations between $a$, $b$, and $c$. The only scenario where this equation has infinitely many solutions is when -ab is not square and positive. In that case, the equation can be rewritten as $x_1^2-dy^2 = N$, with $x_1 = ax$, $d=-ab$ and $N=ac$. Integer solutions to this equation give an upper bound to solutions of the original equation. This is the generalized Pell equation, whose solutions are well known. They are of the form $r_i u_i^n$ for $1\leq i \leq m$ where there are finitely many base solutions $r_i$, and all further solutions are generated by multiplying by units $u_i$ in $\mathbb{Z}[\sqrt d]$. Thus $$\# (p(\mathbb{N}) \bigcap q(\mathbb{N}) \bigcap \{ s, s+1, \cdots n\}) \leq \sum_{i=1}^m \frac{\log(n)}{\log(u_i)},$$ so 
$$\varlimsup_{n-s \to \infty} \de(p,q,s,n) \leq \varlimsup_{n-s \to \infty} \frac{\sum_{i=1}^m \frac{\log(n)}{\log(u_i)}}{n^{\frac{1}{k}}-s^{\frac{1}{k}}}=0.$$
\end{proof}

With this lemma we can now construct a saturated sparsely intersecting set of polynomials.
$\mathscr{d}$
\begin{proof}[Proof of \refl{iscomplete}]
Start by ordering $\mathbb{Z}[X]$ as follows. First, list all polynomials of degree less than or equal to $1$ with coefficients whose absolute values are less than or equal to 1. Then list all polynomials of degree at most $2$ and coefficients with absolute values at most $2$ ordered lexicographically, removing any repeated polynomials. At step $n$, list all polynomials of degree at most $n$ and coefficients with absolute values at most $n$ ordered lexicographically. In this way, we create a bijection between the natural numbers and $\mathbb{Z}[X]$. If $p_i \approx p_j$ for any $j<i$, then there exist $q$ and $\mu$, $\lambda$ such that $p_j = q \circ \mu$ and $p_i = q \circ \lambda$. Replace $p_j$ by $q$, remove any other instances of $q$ in the ordering, and remove $p_i$. Let $P_i$ be the result of this operation completed for $p_1, p_2, \cdots p_i$. Then $P = \bigcap P_i$ is our desired indexed set.
\end{proof}

%We construct this set to show that the results that will be proved are non-trivial.

%\subsection{Main Theorems}

%The goal of this sect will be to prove the following two theorems. %reword, say these are two main theorems

% should we change to mean we can find algorithm explicitly v just that there exists an algorithm
%nvm, just say exists algorithm, numbers and sequences are computable

\subsection{Explicit asymptotic distribution functions and polynomials}
In order to construct computable members of $\Phi'_{Q,P,F}$ we will need to better understand explicit families of adfs and polynomials.
\begin{lem}\labl{approxadf}
If $\{ f_{m,r} \}$ is an explicit linear family of adfs, then there exists a sequence of explicit linear families of continuous adfs $\{ g_{n,m,r} \}$ so that $g_{n,m,r}$ converges to $f_{m,r}$ pointwise.
\end{lem}
\begin{proof}
We mimic the proof in \cite{KuN} that for any adf $f$, there exists a sequence of continuous adfs converging to $f$.
Define increasing sequences $a^n = \{a^n_i\}_{i=1}^\infty$ such that $a^n$ contains $\frac{i}{n+1}$ for $0\leq i \leq n$ and all $t$ so that
\begin{equation}\labeq{overnjump}
\lim_{x \to t^+} f_{m,r}(x) - \lim_{x \to t^-} f_{m,r}(x)>\frac{1}{n}.
\end{equation}
There are finitely many $t$ that satisfy \refeq{overnjump}, so $a^n$ is a finite sequence. Each element of $a^n$ is a computable real number as well. Note that $a^n_{i+1}-a^n_i < \frac{1}{n}$.

Let $g_{n,m,r}\pr{a^n_{i}} = f_{m,r}\pr{a^n_{i}}$ and piecewise linear between $a^n_{i}$ and $a^n_{i+1}$. Then $g_{n,m,r}$ is continuous, non-decreasing, $g_{n,m,r}(0)=0$, and $g_{n,m,r}(1)=1$. Note that for any $a^n_i$, since $\{f_{m,r}\}$ is a linear family of adfs, we have that 
$$
g_{n,m,r}(a^n_i) = \frac{1}{d} \sum_{i=0}^{d-1} g_{n,md, mi+r}(a^n_i)
$$
for all $d$. As $g_{n,m,r}$ is piecewise linear, we have that this equality holds for all $x \in [0,1]$. So $\{g_{n,m,r}\}$ is a linear family of adfs. As $g_{n,m,r}$ is continuous, we only need to check that $g_{n,m,r}(q)$ and $\inf \pr{g^{-1}_{n,m,r}(q)}$ are computable real numbers for all $q \in \mathbb{Q}$. We have that $a^n_i \leq q < a^n_{i+1}$ for some $i$, so 
$$
g_{n,m,r}(q) = \frac{f_{m,r}(a^n_{i+1})-f_{m,r}(a^n_{i})}{a^n_{i+1}-a^n_i} (q-a^n_i)+f_{m,r}(a^n_i)
$$ 
since $g_{n,m,r}$ is piecewise linear.
But the set of computable real numbers is closed under the usual field operations of the reals (see \cite{Weihrauch}), so the real number $g_{n,m,r}(q)$ is computable. If $g^{-1}_{n,m,r}(q)$ consists of a single point, say $r$, then we have that $a^n_{i} \leq r \leq a^n_{i+1}$ and 
$$
r = \pr{q-f_{m,r}(a^n_i)}\frac{a^n_{i+1}-a^n_i}{f_{m,r}(a^n_{i+1})-f_{m,r}(a^n_{i}))}+a^n_i,
$$ 
which is a computable real number. If $g^{-1}_{n,m,r}(q)$ does not consist of a single point, we have that there are maximum and minimum integers $i$ and $j$ such that $g_{n,m,r}\pr{a^n_i} = g_{n,m,r}\pr{a^n_{j}} = q$. Since $g_{n,m,r}\pr{a^n_{i-1}} \neq q$, for any $a^n_{i-1} \leq x < a^n_i$ we have that $g_{n,m,r}(x)<q$. So the real number $\inf\br{g^{-1}_{n,m,r}(q)} = a^n_i$  is computable. Thus $g_{n,m,r}$ is an explicit linear family of adfs.

To see that $g_{n,m,r}$ converges pointwise to $f_{m,r}$, let $t \in [0,1]$.  If $t$ is a discontinuity of $f_{m,r}$, then $\lim_{x \to t^+} f_{m,r}(x) - \lim_{x \to t^-} f_{m,r}(x)>0$, which implies that for some $n_0$ we have that $\lim_{x \to t^+} f_{m,r}(x) - \lim_{x \to t^-} f_{m,r}(x)>\frac{1}{n_0}$. Then $g_{n,m,r}(t)=f_{m,r}(t)$ for $n>n_0$. Now suppose $t$ is not a discontinuity of $f_{m,r}$. Let $\epsilon>0$. Then for some $n_0$, if $x \in \pr{t-\frac{1}{n_0} , t+\frac{1}{n_0}}$, then $f_{m,r}(x) \in \pr{f_{m,r}(t)-\epsilon, f_{m,r}(t)+\epsilon}$. For some $i$, we have that $a^{n_0}_i \leq t \leq a^{n_0}_{i+1}$. As $a^{n_0}_{i+1}-a^{n_0}_i <\frac{1}{n_0}$, we know that $a^{n_0}_i, a^{n_0}_{i+1} \in \pr{t-\frac{1}{n_0} ,t+\frac{1}{n_0}}$. Thus $f_{m,r}\pr{a^{n_0}_i}>f_{m,r}\pr{t}-\epsilon$ and $f_{m,r}\pr{a^{n_0}_{i+1}}<f_{m,r}(t)+\epsilon$. But since $g_{n_0,m,r}\pr{a^{n_0}_i}=f_{m,r}\pr{a^{n_0}_i}$ for all $i$ and $g_{n_0,m,r}\pr{a^{n_0}_i} \leq g_{n_0,m,r}(t)\leq g_{n_0,m,r}\pr{a^{n_0}_{i+1}}$, then we have that $f_{m,r}(t)-\epsilon<g_{n_0,m,r}(t)<f_{m,r}(t)+\epsilon$. Thus $g_{n,m,r}$ converges pointwise to $f_{m,r}$.
\end{proof}

%\subsection{Explicit polynomials}
The following result by S. Tengely \cite{Tengely} is useful for constructing an explicit set of polynomials and proving \refmt{Main}.
\begin{thrm}[S. Tengely]\labt{Tengely}
Let $p,q \in \mathbb{Z}[X]$ be monic polynomials with $\deg p = n \leq \deg q = m$ such that $p(X)-q(Y)$ is irreducible in $\mathbb{Q}[X,Y]$ and $\gcd(n,m)>1$. Let $d>1$ be a divisor of $\gcd(n,m)$. If $(x,y) \in \mathbb{Z}^2$ is a solution of the Diophantine equation $p(x)=q(y)$, then
\begin{equation*}
\max \{|x|,|y| \} \leq d^{\frac{2m^2}{d}-m}\left (m+1\right )^\frac{3m}{2d} \left (m/d+1\right )^\frac{3m}{2} \left (h+1\right )^{\frac{m^2+mn+m}{d}+2m},
\end{equation*}
where $h = \max \left \{ H(p),H(q)\right \}$ and $H(\cdot)$ denotes the maximum of the absolute values of the coefficients.
\end{thrm}
 We prove the following theorem to show that the second part of \refmt{Main} is not vacuous.

\begin{lem}
There is a sparsely intersecting explicit set of polynomials that contains $p(X)=X$.
\end{lem}
\begin{proof}
We proceed with the construction in the proof of \refl{iscomplete} exactly as before, creating an ordering of all polynomials. Note that by construction, $p_1$ is the identity polynomial. As before, at step $i>1$ we check if $p_i$ and $p_j$ satisfy the properties of \reft{Tengely} for $1<j<i$. If they do, then there is a computable  bound $M$ on the absolute values of solutions to $p_i(x)=p_j(y)$. Thus if we set $N(m) = \frac{M}{m}$, we have that if $n-s>N(m)$, then $\de(p_i, p_j, s, n) < m$ since $\de(p_i, p_j, s, n) < \frac{M}{n}$.

This ensures sparse intersection and explicitness when we do not consider $p_1$. Suppose $j=1<i$ and $p_i(n)=m$. Writing $p_i(n) = a_k n^k + \cdots + a_0$, set $a^* = \max \{ a_k, \cdots, a_0 \}$. Note that if $n> \frac{2ka^*}{a_k}$, then $p_i(n) > \frac{1}{2} a_k n^k$. Thus 
$$
\de(p_i,p_1,s,n) \leq \frac{2ka^*}{a_k n}+\sqrt{\frac{2}{a_k n}}
$$ 
for $n > \frac{2ka^*}{a_k}$. Set $N(m) =\ceil{ \max \left\{ \frac{4ka^*m}{a_k}, \frac{8m^2}{a_k} \right\} }$. Therefore $N(m)$ is a computable sequence. Moreover, if $n-s > N(m)$, then $d_{p_i, p_1, s, n} < \frac{1}{2m} + \frac{1}{2m} = \frac{1}{m}$.

If $p_i$ and $p_j$ do not satisfy the conditions of \reft{Tengely} for any $ 1 < j < i$, we remove $p_i$ from $P$ and relabel $p_{i+1}$ to $p_i$ and so on. Let $P_i$ be the resulting set of this procedure conducted for $1\leq j \leq i$. Set $P = \bigcap P_i$.
\end{proof}

\subsection{Homogeneous Moran set structure}
We will construct a subset $\Phi'_{Q,P,F}$ of $\Phi_{Q,P,F}$ so that $\Phi'_{Q,P,F}$ has the structure of a homogeneous Moran set and full Hausdorff dimension. Let $\{ n_k \}$ be a sequence of positive integers and $\{c_k\}$ be a sequence of positive numbers such that $n_k \geq 2$, $0<c_k<1$, $n_1 c_1 \leq d$, and $n_k c_k \leq 1$, where $d$ is a positive real number. For any $k$, let $D_k = \{ (i_1, \cdots, i_k): 1\leq i_j \leq n_j, 1\leq j \leq k \}$, and $D = \bigcup D_k$, where $D_0 =\emptyset$. If $\sigma = ( \sigma_1, \cdots , \sigma_k) \in D_k$, $\tau = (\tau_1 ,\cdots , \tau_m) \in D_m$, put $\sigma * \tau = (\sigma_1, \cdots , \sigma_k, \tau_1, \cdots , \tau_m)$.

\begin{definition}
Suppose $J$ is a closed and bounded interval.  The collection of closed subintervals $ \mathcal{F} = \{ J_\sigma : \sigma \in D\}$ of $J$ has \textbf{homogeneous Moran structure}~if:
\begin{enumerate}
	\item $J_{\emptyset} = J$;
	\item $\forall k \geq 0, \sigma \in D_k, J_{\sigma *1}, \cdots , J_{\sigma * n_{k+1}}$ are subintervals of $J_\sigma$ and $\mathring{J}_{\sigma*i}\cap \mathring{J}_{\sigma*j}=\emptyset$ for $i \neq j$;
	\item $\forall k \geq 1, \forall \sigma \in D_{k-1}, 1\leq j \leq n_k$, $c_k = \frac{\lambda(J_{\sigma*j})}{\lambda(J_\sigma)}$.
\end{enumerate}
\end{definition}

Suppose that $\mathcal{F}$ is a collection of closed subintervals of $J$ having homogeneous Moran structure. Let $E(\mathcal{F}) = \bigcap_{k\geq 1} \bigcup_{\sigma \in D_k} J_\sigma$. We say $E(\mathcal{F})$ is a \textbf{homogeneous Moran set determined by} $\mathcal{F}$, or it is a \textbf{homogeneous Moran set determined by} $J$, $\{ n_k \}$, $\{ c_k \}$.  We will need the following theorem of D. Feng, Z. Wen, and J. Wu from \cite{FengWenWu}.

\begin{thrm}[D. Feng, Z. Wen, and J. Wu]
If $S$ is a homogeneous Moran set determined by $J$, $\{n_k \}$, $\{ c_k \}$, then
\begin{equation*}
\dimh{S} \geq \liminf_{k \to \infty} \frac{\log(n_1 n_2 \cdots n_k)}{-\log(c_1 c_2 \cdots c_{k+1} n_{k+1})}.
\end{equation*}
\end{thrm}

\subsection{The construction}\labs{construction}
The construction given in this section will be related to the constructions given by the second author in \cite{Mance7, Mance3}.
%Suppose we are given an explicit tuple $(Q,P,F)$.
Suppose that we are given a computable and computably growing basic sequence $Q = \{q_n\}$, an explicit set of sparsely intersecting polynomials $P = \{p_n\}$%with $p_1(n) = n$
, and a set of explicit linear families of upper and lower adfs $F$ defined by \refeq{Fadf}.
%$F = \br{\overline{f}_{p_i,m,r}, \underline{f}_{p_i,m,r} : m \in \mathbb{N} \text{ and } 0 \leq r < m}$ for each $p_i \in P$.
Let 
$$
\br{\br{ \overline{g}_{n,p,m,r} }_{m \in \mathbb{N}, 0\leq r < m}}_{n=1}^\infty \hbox{ and } \br{\br{  \underline{g}_{n,p,m,r} }_{m \in \mathbb{N}, 0\leq r < m}}_{n=1}^\infty
$$ 
be sequences of explicit linear families of continuous adfs where $\{\overline{g}_{n,p,m,r}\}_{n=1}^\infty$ and   $\{\underline{g}_{n,p,m,r}\}_{n=1}^\infty$ converge pointwise to $\overline{f}_{p,m,r}$ and $\underline{f}_{p,m,r}$, respectively. For 
$$
I(j,k)=\left[\frac{j} {k},\frac {j+1} {k}\right),
$$ 
we define
\begin{align*}
\Delta_k = \min_{\substack{1 \leq l \leq k\\ 0\leq r < k!\\ 0 \leq j < k}}  \left \{ \min  \left \{ \lambda \left (\overline{g}_{k, p_l, k!, r}^{-1} \big (I(j,k) \right ),  \lambda \left ( \underline{g}_{k, p_l, k!, r}^{-1} \big ( I(j,k)\big ) \right ) \right \}   \right \}
\end{align*}
with $\lambda$ the Lebesgue measure.  Note that 
$$
\lambda \left (\overline{g}_{k, p_l, k!, r}^{-1} \left ( I(j,k)\right ) \right ) 
= \sup \overline{g}_{k,p_l,k!,r}^{-1} \left ( \frac{j+1}{k}\right ) -\inf  \overline{g}_{k,p_l,k!,r}^{-1} \left ( \frac{j}{k}\right ),
$$
so is a computable real number as it is the difference of two computable real numbers. This means that $\Delta_k$ is a computable real number since we are taking finitely many maximums and minimums of computable real numbers.  Define
\begin{align*}
\epsilon_k &= \frac{\min \br{ \log(q_k)^{\frac{1}{2}}, \log(q_1 \cdots q_{k-1})^{\frac{1}{2}}}}{\log(q_k)};\\
\nu_{j,1} &= \min \left \{t \in \mathbb{N} : \min \left \{ \log(q_k)^{\frac{1}{2}}, \log(q_1 \cdots q_{k-1})^{\frac{1}{2}}\right \} \geq \log(4)-\log( \Delta_j) ,\forall k \geq t ,     \right \};\\
\nu_{j,2} &= \min \left \{ t \in \mathbb{N} : \de(p_k,p_l,1,n) < \frac{1}{(j!)^2 j^3} \forall 1\leq l<k \leq j, \forall n>t  \right \}.
\end{align*}
%Recall that $d_{p,q,s,n}$ is defined in \refeq{dpqsn}.  
We have that $\br{\nu_{j,1}}$ is a computable sequence as $\Delta_k$ is a computable real number and $Q$ is a computably growing basic sequence. We also have that $\br{\nu_{j,2}}$ is a computable sequence since $P$ is an explicit sparsely intersecting set of polynomials and by \reft{Tengely}. Finally, set
\begin{equation*}
\nu_j = \max \{ \nu_{j,1}, \nu_{j,2} \}.
\end{equation*}
We will define sequences of integers $\{l_j\}$ and $\{L_j\}$ inductively. Set
\begin{align*}
& l_1 = \max \{ \nu_2-1, 1\}; \\
& \xi_{j,k,l,t} = \de(p_k,p_l,L_{j-1}+j!jt,L_{j-1}+j!jt+n);\\
&\psi_j = \min\bigg\{t \in \mathbb{N} : \xi_{j,k,l,t}
 < \frac{1}{(j+1)!^2 (j+1)^3} \ \forall  n\geq (j-1)! t \wedge k,l < j+1 \bigg\}; \\
&l_j =  \max \left\{\min \left \{ t \in \mathbb{N} : L_{j-1} + j! j t \geq \nu_j-1\right \} , \psi_j, j^2 \right \};\\
&L_j = \sum_{i=1}^j i!i l_j.
\end{align*}

Clearly, the sequence $\{\nu_j\}$ is computable since the sequences $\br{\nu_{j,1}}$ and $\br{\nu_{j,2}}$ are computable. The sequence $\br{\psi_j}$ is computable since $P$ is explicit.  Thus the sequences $\{l_j\}$ and $\{L_j\}$ are also computable.

Let 
$$
U = \{ (i,b,c,d) \in \mathbb{N}^4 : b\leq l_i, c \leq i, d \leq i!\}.
$$ 
Define $\Xi_l : U \to p_l(\mathbb{N})$ by 
$$
\Xi_l(i,b,c,d) = p_l(L_{i-1}+b i!i+ c i! + d).
$$
It is easy to show that $\Xi_l$ is a bijection.

Put $P'_i = \{ p_{\sigma(j)}\}_{j=1}^{i}$ where $\sigma$ is a permutation of $\{1,\cdots, n\}$ such that if $\sigma(i) < \sigma(j)$, then $p_j \lnsim p_i$. We can find such a $\sigma$ since $P$ is sparsely intersecting. Reorder $P_i$ so that it is equal to $P'_i$.
Define 
\begin{align*}
i(n) &= \max \{ t : n>p_{l}(L_t), \forall l\leq t\};\\
\dg (n) &= \max \br{ j : n \in p_j(\mathbb{N}), p_j \in P_{i(n)}}.
\end{align*}
Moreover, define $b(n)$, $c(n)$, and $d(n)$ by
$$(i(n), b(n), c(n), d(n)) = \Xi_{\dg(n)}^{-1}(n).
$$ 
These functions are defined as $\Xi_l$ is a bijection. Define the sets
\begin{align*}
V^*_{l,n} &= 
\begin{cases}
	[0,q_n) & \text{ if }  l<i \\
	q_n \text{ } \overline{g}_{i(n),l,i(n)!,d(n)}^{-1}\pr{\left[\frac{c(n)}{i(n)}, \frac{c(n)+1}{i(n)}\right)} & \text{ if } i(n) \equiv 0 \pmod 2, l\geq i(n) > 1 \\
	q_n \text{ } \underline{g}_{i(n),l,i(n)!,d(n)}^{-1}\pr{\left[\frac{c(n)}{i(n)}, \frac{c(n)+1}{i(n)}\right)} & \text{ if } i(n) \equiv 1 \pmod 2, l\geq i(n) > 1
\end{cases};\\
V_{l,n} &= \left[\inf(V^*_{l,n}), \inf(V^*_{l,n})+q_{n}\Delta_{i(n)}-q_n-1\right] \cap \mathbb{Z}.
\end{align*}
Note that we can write $V_{\dg(n),n} = \{ \alpha(n), \alpha(n)+1, \cdots , \beta(n) \}$, where $\{\alpha(n)\}$ and $\{\beta(n)\}$ are sequences of integers.  Moreover, 
\begin{align*}
\alpha(n) &= \ceil{\inf\left(V^*_{\dg(n),n}\right)};\\
\beta(n) &= \floor{\inf\left(V^*_{\dg(n),n}\right)+q_n\Delta_{i(n)}-q_n-1}.
\end{align*}
 We will discuss the computability of these sequences in the proof of \refmt{Main}. Set
\begin{equation*}
\Phi'_{Q,P,F} := \br{ x \in [0,1) : E_n\in V_{\dg(n),n}}.
\end{equation*}

We will now work towards proving the second part of \refmt{Main}.
%We will now establish lemmas to prove the second part of \refmt{Main}.

\begin{lem} \labl{contadf}
A sequence $\{x_n \}$ has a continuous adf $f$ if and only if the sequence $\{ f(x_n) \}$ is u.d. mod 1.
\end{lem}
\begin{proof}
Since $f$ is non-decreasing, $x_n\leq \gamma$ if and only if $f(x_n) \leq f(\gamma)$. Thus 
$$
\frac{A_n([0,\gamma], x_n)}{n} = \frac{A_n([0,f(\gamma)], f(x_n))}{n}.
$$ 
But  $\{x_n\}$ has continuous adf $f$, so 
$$
\lim_{n\to \infty} \frac{A_n([\gamma,\gamma], x_n)}{n} = f(\gamma^+)-f(\gamma^-) = 0.
$$ 
Thus 
$$
\lim_{n \to \infty} \frac{A_n([0,\gamma), x_n)}{n} = \lim_{n \to \infty} \frac{A_n([0,\gamma],x_n)}{n}.
$$ 
Similarly we have that 
$$
\lim_{n\to \infty} \frac{A_n([0,f(\gamma)), f(x_n))}{n} = \lim_{n\to \infty} \frac{A_n([0,f(\gamma)], f(x_n))}{n}
$$ 
since $\lambda([\gamma,\gamma]) = 0$. 
Therefore 
$$
\lim_{n \to \infty} \frac{A_n([0,\gamma), x_n)}{n}-f(\gamma) = \lim_{n\to \infty} \frac{A_n([0,f(\gamma)], f(x_n))}{n}-f(\gamma).
$$ 
As $f$ is continuous, it must map $[0,1]$ onto $[0,1]$, so this second limit satisfies the definition for uniform distribution mod $1$. These limits converge to $0$ if and only if the other does, and we are done.
\end{proof}

\begin{lem}\labl{SameADF}
Suppose that $Q$ is a basic sequence that is infinite in limit, $f : \mathbb{N} \to \mathbb{N}$ is an eventually increasing function, and  both $\OQfmrxp{f}$ and $\br{ \frac {E_{f(mn+r)}+1} {q_{f(mn+r)}} }$ have the same upper and lower adfs. Then $\OQfmrx{f}$ has the same upper and lower adfs as $\OQfmrxp{f}$.
\end{lem}
\begin{proof}
Suppose that $g$ is the upper adf of the sequence $\OQfmrx{f}$.
Let $Q_f=\{ q_{f(n)} \}_{n=0}^\infty$. As $f$ is eventually increasing, $Q_f$ is a basic sequence that is infinite in limit. It is clear from the definition of $T_{Q,n}(x)$ that 
$$
\frac{E_{f(n)}}{q_{f(n)}} \leq T_{Q,f(n)}(x) \leq \frac{E_{f(n)}+1}{q_n}.
$$
Hence 
\begin{align*}
\frac{A_n \pr{ [0,\gamma ),\br{ \frac{E_{f(i)}+1}{q_{f(i)}}}}}{n}-g(\gamma) &\leq \frac{A_n\pr{[0,\gamma), \OQfx}}{n}-g(\gamma)\\
 &\leq \frac{A_n\pr{[0,\gamma), \OQfxp}}{n}-g(\gamma).
\end{align*}
We also have that 
$$
\frac{E_{f(n)}+1}{q_{f(n)}}-\frac{E_{f(n)}}{q_{f(n)}} = \frac{1}{q_{f(n)}}
$$ 
which goes to $0$ as $n$ goes to $\infty$. But $\frac{E_{f(n)}+1}{q_{f(n)}}$ and $\OQfx$ have the same upper adf by assumption.
Thus for all $\gamma$ we have that
$$
\varlimsup_{n \to \infty} \frac{A_n\pr{[0,\gamma), \OQfx}}{n}-g(\gamma)=0,
$$
so $\OQfx$ has the same upper adf as $\OQfxp$. The proof for the lower adf is identical.
\end{proof}

We have approximated each adf $f$ by a sequence of continuous adfs $g_n$ as in \refl{approxadf}. To show $\OQfmrx{p}$ has adf $f$ we will use \refl{contadf} and \refl{SameADF} as well as the following definition and theorem. 

\begin{definition}
Let $\omega = \{x_i \}$ be a sequence of real numbers. The \textbf{upper discrepancy} with respect to adf $f$ of $\omega$ is
$$
\overline{D}_n^f(\omega) := \max \br{ \sup_\gamma \br{ \frac{A_n([0,\gamma),x_i)}{n}-f(\gamma)},0}.
$$ 
The \textbf{lower discrepancy} of $\omega$ is 
$$
\underline{D}_n^f(\omega) := \min \br{ \inf_\gamma \br{ \frac{A_n([0,\gamma),x_i)}{n}-f(\gamma)},0}.
$$
\end{definition}
Any theorems or lemmas that hold for the upper discrepancy will also hold for the lower discrepancy with a sign change, so from here on we only state and prove for the upper discrepancy. Many of the properties of the upper and lower discrepancies are the same for the discrepancy, and the proofs are nearly identical.

\begin{thrm}\labt{discrep}
The following properties of the upper and lower discrepancies hold for any sequence of real numbers $\omega = \{z_i\}$ and any continuous adf $f$.
\begin{enumerate}
\item The inequality $\uD{n}{f}(\omega) \leq D_n^f(\omega)$ holds.
\item If $\omega_i$ a sequence of real numbers of length $|\omega_i |$, then
$$\uD{n}{f}(\omega_1 \omega_2 \cdots \omega_j) \leq \frac{\sum_{i=1}^j |\omega_i| \uD{|\omega_i|}{f}(\omega_i)}{\sum_{i=1}^j |\omega_i|}.$$
\item If $\lim_{n\to \infty} \uD{n}{f}(\omega) = 0$, then the upper adf $\overline{f}$ of $\omega$ is bounded above by $f$.
\item If $\{f_{m,r}\}$ is a linear family of adfs, then for all $d$, $m$, and $r$, we have that $$\uD{n}{f_{m,r}}(\omega) \leq \max_{0\leq i < m} \br{ \uD{n}{f_{dm, mi+r}}(\omega)}+\frac{m(d+1)}{n-md}.$$
\item If $\overline{f}$ and $\underline{f}$ are the upper and lower adfs of $\omega$ then
$\overline{D}_{n}^{\overline{f}}(\omega) \leq~\overline{D}_{n}^{\underline{f}}(\omega)$.
\item If $\omega$ is a non-decreasing sequence, then 
$$
\overline{D}^f_n(\omega) \leq \max \br{ \abs{f(z_i)-\frac{i}{n}}, \abs{f(z_i)-\frac{i+1}{n}} }.
$$
\end{enumerate}
\end{thrm}
\begin{proof}
\ 
\begin{enumerate}
\item We have that 
\begin{align*}
\max \br{ \sup_\gamma \br{ \frac{A_n([0,\gamma),z_i)}{n}-f(\gamma)},0} 
&\leq \abs{ \sup_\gamma \br{ \frac{A_n([0,\gamma),z_i)}{n}-f(\gamma)} } \\
&\leq \sup_\gamma \abs{\frac{A_n([0,\gamma),z_i)}{n}-f(\gamma) },
\end{align*}
 which implies that $\overline{f}(\gamma) \leq f(\gamma)$ for all $\gamma$.

\item The proof is identical to the proof of Theorem 2.6 in Chapter 2 of \cite{KuN}.

\item We have that 
\begin{equation*}
\lim_{n \to \infty} \max \br{ \sup_\gamma \br{ \frac{A_n([0,\gamma),\omega)}{n}-f(\gamma)},0 }= 0
\end{equation*}
which implies that 
\begin{align*}
\lim_{n \to \infty} \sup_\gamma \br{ \frac{A_n([0,\gamma),\omega)}{n}-f(\gamma)} \leq 0 = \lim_{n \to \infty} \sup_\gamma \br{ \frac{A_n([0,\gamma),\omega)}{n}-\overline{f}(\gamma)}. \end{align*} But this implies that $\overline{f}(\gamma) \leq f(\gamma)$ for all $\gamma$.

\item Note that 
\begin{align*}
&d \sup_\gamma \frac{A_{\floor{n/m}}([0,\gamma ),\br{z_{mi+r}})}{\floor{\frac{n}{m}}} - f_{m,r}(\gamma) \\
&\leq \sup_\gamma \sum_{j=0}^{d-1} \frac{A_{\floor{\frac {n}{md}}}([0,\gamma), \br{z_{dmi+mj+r}})+1+d}{\floor{\frac{n}{md}}} -d f_{m,r}(\gamma)\\
&= \sup_\gamma \sum_{j=0}^{d-1} \frac{A_{\floor{\frac {n}{md}}}([0,\gamma), \br{z_{dmi+mj+r}})}{\floor{\frac{n}{md}}} - \sum_{j=0}^{d-1} f_{md, mj+r}(\gamma) +\frac{md(d+1)}{n-md}\\
&= \frac{md(d+1)}{n-md}+\sum_{i=0}^{d-1} \uD{n}{f_{md, mi+r}}(\omega) 
\leq \frac{md(d+1)}{n-md}+d \max_{0\leq i < m} \uD{n}{f_{dm, mi+r}}(\omega).
\end{align*}
Dividing by $d$ yields the result.

\item Since $\underline{f} \leq \overline{f}$, we have that 
$$
\sup \left\{ \frac{A_n([0,\gamma),z_i)}{n}-\overline{f}(\gamma)\right\} \leq \sup \left\{ \frac{A_n([0,\gamma),z_i)}{n}-\underline{f}(\gamma)\right\}.
$$

\item The proof is identical to the proof of Theorem 1.4 in Chapter 2 of \cite{KuN}.
\end{enumerate}
\end{proof}

Define 
$$
\upsilon (n) := \# V_{\dg(n),n} = \alpha(n)-\beta(n)+1.
$$
\begin{lem}\labl{moranNumber}
For all natural number $n$ we have $\upsilon(n) > 4 q_k^{1-\epsilon_k}-2\geq 2$.
\end{lem}
\begin{proof}
The interval $V_{\dg(n),n}$ is  of length $q_n \Delta_{i(n)}-1$. Thus $\upsilon(n) \geq q_n \Delta_{i(n)} - 2$. But $q_k > 4\Delta_k^{-\frac{1}{\epsilon_k}}$, so $4 q_k^{-\epsilon_k} < \Delta_k$. Thus $\upsilon(n) > 4q_k^{1-\epsilon_k}-2$. But $\epsilon_k\leq 1$, so $4q_k^{1-\epsilon_k}-2 \geq 4-2 = 2$.
\end{proof}

%CHECK COLON EQUALS

We will use the following basic lemma.

\begin{lem}\labl{tcorr}
Let $L$ be a real number and $(a_n)_{n=1}^{\infty}$ and $(b_n)_{n=1}^{\infty}$ be two sequences of positive real numbers such that
$
\sum_{n=1}^{\infty} b_n=\infty \hbox{ and } \lim_{n \to \infty} \frac {a_n} {b_n}=L.
$
Then
$$
\lim_{n \to \infty} \frac {a_1+a_2+\cdots+a_n} {b_1+b_2+\cdots+b_n}=L.
$$
\end{lem}

We now prove the second part of \refmt{Main}.  The proof that the set $\Phi_{Q,P,F}$ has full Hausdorff dimension is done similarly.

\begin{proof}[Proof of \refmt{Main}]

Let $x \in \Phi'_{Q,P,F}$, $p \in P$, and $m\in \mathbb{N}$, $0\leq r <m$. The sequence $\OQfmrxp{p}$ can then be written as
$$
\OQfmrxp{p}=Y_{1,0} \cdots Y_{1,l_1-1} Y_{2,0} Y_{2,1} \cdots Y_{2,l_2-1} \cdots,
$$
 where each $Y_{i,k}$ is a block of digits with length $\frac{i!i}{m}$ when $i\geq m$. When $m<i$, the length of $Y_{i,k}$ is less  than $\frac{i!i}{m}$. We will look at the discrepancies of the blocks $Y_{i,k}$, and apply \reft{discrep} to show the claim.

 Put $X_j = Y_{j,0} \cdots Y_{j, l_j-1}$. By the definition of the upper discrepancy, we have that 
$$
\overline{D}_{|Y_{i,k}|}^{\overline{g}_{i,p,m,r}}(Y_{i,k})<\frac{1}{i},
$$ 
unless $k \in \{0, l_i-1\}$, in which case the upper discrepancy is bounded by 1. We must also consider when the digits of $x$ lie along sequences of the form $\{q(n)\}$ for some other $q \in P$. For $L_{i-1} \leq n \leq L_i$, there are at most $i!i$ polynomials that could intersect 
$p$, and since $\max d_{p_k,p_l,n} < \frac{1}{i^3 i!^2}$ at stage $i$, we have that the number of terms that increase the 
discrepancy is at most $\frac{i!i l_i }{i! i^2} = \frac{l_i}{i}$. Because $\overline{g}_{n,p,m,r}$ does not converge to 
$\overline{f}_{p,m,r}$ uniformly, we cannot measure the discrepancy of the sequence with respect to $\overline{f}_{p,m,r}$ 
directly. Instead, we use the fact of pointwise convergence and discrepancy with respect to $\overline{g}_{n,p,m,r}$. Defining 
$k(n) = n-L_{i(n)}$, we can write $k(n) = a(n)i(n)!i(n)+ b(n)$, and 
\begin{align*}
&\varlimsup_{n \to \infty} \left(\frac{A_n\pr{[0,\gamma), \OQfmrxp{p}}}{n}-\overline{f}_{p,m,r}(\gamma) \right) \\
= &\varlimsup_{n \to \infty} \frac{\sum_{j=1}^{i(n)-1} A_{|X_j|}([0, \gamma), X_j)+A_{k(n)}([0,\gamma), X_{i(n)})}{\sum_{j=1}^{m-1}\lfloor \frac{j!j l_j}{m}\rfloor + \sum_{j = m}^{i(n)-1} \frac{j!jl_j}{m}+k(n)}\\
&- \frac{\sum_{j=1}^{i(n)-1} \frac{j!jl_j}{m}\overline{f}_{p,m,r}(\gamma)+k(n)\overline{f}_{p,m,r}(\gamma)}{\sum_{j=1}^{m-1}\lfloor \frac{j!j l_j}{m}\rfloor + \sum_{j = m}^{i(n)-1} \frac{j!jl_j}{m}+k(n)} \\
\leq &\varlimsup_{n \to \infty} \frac{\sum_{j=1}^{i(n)-1} \pr{ A_{\frac{j!jl_j}{m}}([0,\gamma), X_j)-\frac{j!jl_j}{m}\overline{g}_{j,p,m,r}(\gamma)}}{\sum_{j=1}^{m-1}\floor{ \frac{j!j l_j}{m}} + \sum_{j = m}^{i(n)-1} \frac{j!jl_j}{m}+k(n)}\\
&+ \frac{A_{k(n)}\pr{[0,\gamma ), X_{i(n)}}-k(n)\overline{g}_{i(n),p,m,r}(\gamma)}{\sum_{j=1}^{m-1}\floor{ \frac{j!j l_j}{m}} + \sum_{j = m}^{i(n)-1} \frac{j!jl_j}{m}+k(n)}\\
&+\varlimsup_{n \to \infty} \frac{\sum_{j=1}^{i(n)-1} \pr{ \frac{j!jl_j}{m} \overline{g}_{j,p,m,r}(\gamma)-\frac{j!jl_j}{m} \overline{f}_{p,m,r}(\gamma) }}{\sum_{j=1}^{m-1}\floor{ \frac{j!j l_j}{m}} + \sum_{j = m}^{i(n)-1} \frac{j!jl_j}{m}+k(n)}\\
&+\frac{+k(n)\overline{g}_{i(n),p,m,r}(\gamma)-k(n)\overline{f}_{p,m,r}(\gamma)}{\sum_{j=1}^{m-1}\floor{ \frac{j!j l_j}{m}} + \sum_{j = m}^{i(n)-1} \frac{j!jl_j}{m}+k(n)}\\
&= S_1 + S_2+S_3
\end{align*}
where
\begin{align*}
S_1 : = &\varlimsup_{n \to \infty} \frac{\sum_{j=1}^{m-1} \ceil{ \frac{j!j l_j}{m}} + \sum_{j=m}^{i(n)-1} \frac{j j!(l_{j}-2)}{m}\frac{j-1}{j}(\frac{1+m}{j})}{\sum_{j=1}^{m-1}\floor{ \frac{j!j l_j}{m}} + \sum_{j = m}^{i(n)-1} \frac{j!jl_j}{m}+a(n)i(n)!i(n)+b(n)}\\
&+\frac{\sum_{j=m}^{i(n)-1} \frac{j j!l_{j}}{m}\frac{1}{j}+\sum_{j=m}^{i(n)-1} \frac{2j!j}{m}}{\sum_{j=1}^{m-1}\floor{ \frac{j!j l_j}{m}} + \sum_{j = m}^{i(n)-1} \frac{j!jl_j}{m}+a(n)i(n)!i(n)+b(n)}\\
S_2 : = &\varlimsup_{n \to \infty} \frac{\frac{l_{i(n)-1}}{i(n)^2} i(n)!i(n)+\left(a(n)-\frac{l_{i(n)-1}}{i(n)^2}\right)i(n)!i(n)\frac{1}{i(n)}+b(n)}{\sum_{j=1}^{m-1}\floor{ \frac{j!j l_j}{m}} + \sum_{j = m}^{i(n)-1} \frac{j!jl_j}{m}+a(n)i(n)!i(n)+b(n)}\\
S_3 : = &\varlimsup_{n \to \infty} \frac{\sum_{j=1}^{i(n)-1} \frac{j!jl_j}{m} \overline{g}_{j,p,m,r}(\gamma)-\frac{j!jl_j}{m} \overline{f}_{p,m,r}(\gamma)}{\sum_{j=1}^{m-1}\floor{ \frac{j!j l_j}{m}} + \sum_{j = m}^{i(n)-1} \frac{j!jl_j}{m}+k(n)} \\
&+ \frac{  \frac{k(n)}{i(n)}(\overline{g}_{i(n),p,m,r}(\gamma)-\overline{f}_{p,m,r}(\gamma) )}{\sum_{j=1}^{m-1}\floor{ \frac{j!j l_j}{m}} + \sum_{j = m}^{i(n)-1} \frac{j!jl_j}{m}+k(n)}
\end{align*}
For this first sum $S_1$, we find that
\begin{align*}
&\varlimsup_{n \to \infty} \frac{\sum_{j=1}^{m-1} \ceil{ \frac{j!j l_j}{m}} + \sum_{j=m}^{i(n)-1} \frac{j j!(l_{j}-2)}{m}\frac{j-1}{j}\frac{1+m}{j}+\sum_{j=m}^{i(n)-1} \frac{j j!l_{j}}{m}\frac{1}{j}+\sum_{j=m}^{i(n)-1} \frac{2j!j}{m}}{\sum_{j=1}^{m-1}\floor{ \frac{j!j l_j}{m}} + \sum_{j = m}^{i(n)-1} \frac{j!jl_j}{m}+a(n)i(n)!i(n)+b(n)}\\
=& \varlimsup_{n \to \infty} \frac{\sum_{j=1}^{i(n)-1} (m j!(l_{j}-2)+j!l_{j}+2j!j)}{\sum_{j = 1}^{i(n)-1}( j!jl_j)+a(n)i(n)!i(n)+b(n)}\\
\leq & \varlimsup_{n \to \infty} \frac{ m(i(n)-1)!(l_{i(n)-1}-2)+(i(n)-1)!l_{i(n)-1}+2(i(n)-1)!(i(n)-1)}{i(n)!i(n)l_{i(n)}}\\
\leq & \varlimsup_{n \to\infty} \frac{m}{i(n)}+\frac{1}{l_{i(n)}} = 0.
\end{align*}

For the second term $S_2$ we have that
\begin{align*}
&\varlimsup_{n \to \infty} \frac{\frac{l_{i(n)-1}}{i(n)^2} i(n)!i(n)+\left(a(n)-\frac{l_{i(n)-1}}{i(n)^2}\right)i(n)!i(n)\frac{1}{i(n)}+b(n)}{\sum_{j=1}^{m-1}\floor{ \frac{j!j l_j}{m}} + \sum_{j = m}^{i(n)-1} \frac{j!jl_j}{m}+a(n)i(n)!i(n)+b(n)}\\
=& \varlimsup_{n \to \infty} \frac{(i(n)-1)!(i(n)-1)l_{i(n)-1}\frac{1}{i(n)} + a(n)i(n)!+b(n)}{\sum_{j = 1}^{i(n)-1}( j!jl_j)+a(n)i(n)!i(n)+b(n)}\\
= & \varlimsup_{n \to \infty} \frac{ (i(n)-1)!(i(n)-1)l_{i(n)-1}\frac{1}{i(n)} + a(n)i(n)!+b(n)}{(i(n)-1)!(i(n)-1)l_{i(n)-1}+a(n)i(n)!i(n)+b(n)}
\leq  \varlimsup_{n \to\infty} \frac{2}{i(n)} = 0.
\end{align*}

 For the final term $S_3$ we get
\begin{align*}
& \varlimsup_{n \to \infty} \frac{\sum_{j=1}^{i(n)-1} \frac{j!jl_j}{m} \overline{g}_{j,p,m,r}(\gamma)-\frac{j!jl_j}{m} \overline{f}_{p,m,r}(\gamma) + \frac{k(n)}{i(n)}\pr{\overline{g}_{i(n),p,m,r}(\gamma)-\overline{f}_{p,m,r}(\gamma) }}{\sum_{j=1}^{m-1}\floor{ \frac{j!j l_j}{m}} + \sum_{j = m}^{i(n)-1} \frac{j!jl_j}{m}+k(n)}\\
= & \varlimsup_{n \to \infty} \frac{\sum_{j=1}^{i(n)-1} \frac{j!jl_j}{m} \overline{g}_{j,p,m,r}(\gamma)-\frac{j!jl_j}{m} \overline{f}_{p,m,r}(\gamma) + \frac{k(n)}{i(n)}\pr{\overline{g}_{i(n),p,m,r}(\gamma)-\overline{f}_{p,m,r}(\gamma) }}{\sum_{j = 1}^{i(n)-1} \frac{j!jl_j}{m}+\frac{k(n)}{i(n)}}\\
= & \varlimsup_{n \to \infty} \frac{\pr{\frac{i(n)!i(n)l_{i(n)}}{m} + \frac{k(n)}{i(n)}}\pr{\overline{g}_{i(n),p,m,r}(\gamma)-\overline{f}_{p,m,r}(\gamma) }}{\frac{i(n)!i(n)l_{i(n)}}{m} + \frac{k(n)}{i(n)}}\\
= & \varlimsup_{n \to \infty}\overline{g}_{i(n),p,m,r}(\gamma)-\overline{f}_{p,m,r}(\gamma) = 0.
\end{align*}
Thus the sequence $\OQfmrxp{p}$ has adf bounded above by $\overline{f}_{p,m,r}$.

This calculation only shows that the upper adf of the sequence $\OQfmrxp{p}$ is bounded above by $\overline{f}_{k,m,r}$. To show that this is the actual upper adf, we must find a sequence $a_n$ along which 
$$\lim_{n \to \infty}  \frac{A_{a_n}\pr{[0,\gamma), \OQfmrxp{p}}}{a_n}-\overline{f}_{p,m,r}(\gamma)=0.$$
 Let $a_n = L_{2n}$. Note that for the second term in the previous sum, we have that $\lim_{n \to \infty} \left(\overline{g}_{i(a_n),p,m,r}(\gamma)-\overline{f}_{p,m,r}(\gamma)\right) = 0$. So we need only check that the first term $S_1$ goes to $0$, as $a(n)$ and $b(n)$ must be $0$ for all $n$. Thus by \refl{tcorr}
\begin{align*}
& \lim_{n \to \infty}  \left| \frac{A_{a_n}\left([0,\gamma), \OQmrxp\right) }{n} -\overline{f}_{p,m,r}(\gamma)\right| \\ 
\leq & \lim_{n \to \infty} \frac{\sum_{j=1}^{m-1} \left| \ceil{ \frac{j!j l_j}{m}} \right| + \sum_{j=m}^{i(a_n)} \left| \frac{j j!(l_{j}-2)}{m}\frac{j-1}{j}\frac{2}{j} \right| +\sum_{j=m}^{i(a_n)} \left| \frac{j j!l_{j}}{m}\frac{1}{j} \right| +\sum_{j=m}^{i(a_n)}\left| \frac{2j!j}{m} \right|}{\sum_{j=1}^{m-1}\left| \floor{ \frac{j!j l_j}{m}} \right|  + \sum_{j = m}^{i(a_n)} \left| \frac{j!jl_j}{m}\right|} \\
= & \lim_{n \to \infty} \frac{ 2i(a_n)! (l_{i(a_n)}-2)+i(a_n)!l_{i(a_n)}+2i(a_n)!i(a_n)}{i(a_n)!i(a_n)l_{i(a_n)}} \leq \lim_{n \to \infty} \frac{1}{i(a_n)} +\frac{1}{l_{i(a_n)}}= 0.
\end{align*}
This implies that $\OQfmrxp{p}$ has upper adf $\overline{f}_{p,m,r}$. Note that throughout the proof we only used that $\alpha(n) \leq E_n \leq \beta(n)+1$, so $\br{ \frac{E_{p(mn+r)}+1}{q_{p(mn+r}}}$ has the same upper adf as $\OQfmrxp{p}$. Therefore by \refl{SameADF}, the upper adf of $\OQfmrx{p}$ is $\overline{f}_{p,m,r}$. The proof that $\OQfmrxp{p}$ has lower adf $\underline{f}_{p,m,r}$ follows similarly.

To compute the Hausdorff dimension of $\Phi'_{Q,P,F}$, we note that $\Phi'_{Q,P.F}$ is a homogeneous Moran set with $J=[0,1]$, $n_k = \upsilon(k) > 4 q_k^{1-\epsilon_k}-2 \geq 2$ by \refl{moranNumber}, and $c_k = q_k^{-1}$. So
\begin{align*}
&\dimh{\Phi'_{Q,P.F}} \geq \liminf_{n \to \infty} \frac{\log\pr{q_1^{1-\epsilon_1} \cdots q_{n-1}^{1-\epsilon_{n-1}}}}{-\log\pr{\frac{1}{q_1} \cdots \frac{1}{q_n} q_n^{1-\epsilon_n}}} \\
& =  \liminf_{n \to \infty} \frac{\log(q_1 \cdots q_{n-1})}{\log(q_1 \cdots q_{n-1})-\epsilon_n log(q_n)}
 =  \liminf_{n \to \infty} \frac{1}{1-\frac{\epsilon_n log(q_n)}{\log(q_1 \cdots q_{n-1})}} = 1.
\end{align*}

We will now prove that $\{\alpha(n)\}$ and $\{\beta(n)\}$ are computable sequences.  Recall that $\alpha(n) = \ceil{\inf(V^*_{l,n})}$ and $\beta(n) = \floor{\inf(V^*_{l,n})+q_n\Delta_{i(n)}-q_n-1}$.
%We can write $V_{l,n} = \{ \alpha(n), \alpha(n)+1, \cdots , \beta(n) \}$ where $\{\alpha(n)\}$ and $\{\beta(n)\}$ are computable sequences. 
To see these sequences are computable, note that $\{c(n)\}$ and $\{i(n)\}$ are computable sequences as $\{L_n\}$ and $\{p_l(n)\}$ are computable sequences for any polynomial $p_l$. Thus $\{c(n)/i(n)\}$ is a computable sequence. By the explicitness of $\br{\overline{g}_{i(n),l,i(n)!,d(n)}}$ for each $n$, and uniform computability of the discontinuities of $f_{p,m,r}$, we have that $\left\{ \inf \overline{g}^{-1}_{i(n),l,i(n)!,d(n)}\left(\frac{c(n)}{i(n)}\right)\right\}_{n=1}^\infty$ is a computable sequence. Thus we have that $\br{ \alpha(n)}$ is the ceiling of the product of two computable sequences as $Q$ is a computably growing basic sequence. Hence $\br{\alpha(n)}$ is a computable sequence. Similarly, it can be shown that $\br{\beta(n)}$ is a computable sequence.
\end{proof}

\section{Applications}

\subsection{Accumulation points along arithmetic progressions}\labs{arith}

In the same vein as \cite{WangWenXi}, we would like to prove results about the accumulation points of $\OQxp$ and prove \reft{strongerWangWenXi}. It is difficult to say anything about accumulation points along polynomially indexed subsequences of $\OQxp$. Inequivalent polynomials can intersect infinitely often, so the accumulation points of $\OQfmrxp{f}$ could be difficult to find. However, if we only sample along arithmetic progressions, a  classification of the accumulation points of $\OQmrxp$ can be given.

For notational convenience, we define
\begin{definition}
For $f:[0,1] \to [0,1]$, set 
$$
I(f) := \overline{\{ x \in [0,1] : f \text{ is increasing at } x\}}.
$$
\end{definition}

After establishing lemmas, we will prove the following theorem.

\begin{thrm}\labt{thrm3.2}
Given a basic sequence $Q$ that is infinite in limit and a linear family of upper and lower adfs $\{ \overline{f}_{m,r}\}$ and $\{\underline{f}_{m,r} \}$, the set of real numbers $x$ such that $\OQmrx$ has upper and lower adfs $\overline{f}_{m,r}$ and $\underline{f}_{m,r}$ and such that $\OQmrxp$ has accumulation points equal to $I \left( \overline{f}_{m,r} \right) \cup I\left(\underline{f}_{m,r}\right)$ has Hausdorff dimension 1.
\end{thrm}

\begin{lem}\labl{acc}
If $x \in I \left( \overline{f}_{m,r} \right) \cup I\left(\underline{f}_{m,r}\right)$, then $x$ is an accumulation point of $\OQmrxp$.
\end{lem}
\begin{proof}
Suppose that $x \in I \left( \overline{f}_{m,r} \right) \cup I\left(\underline{f}_{m,r}\right)$. If $x \in I\left( \overline{f}_{m,r}\right)$, then either $\overline{f}(x-\epsilon)<\overline{f}(x)$ or $\overline{f}(x)< \overline{f}(x+\epsilon)$ for all $\epsilon>0$. The proof for both of these cases is identical, so suppose that $\overline{f}(x-\epsilon)<\overline{f}(x)$. Then 
$$
\varlimsup_{n \to \infty} \frac{A_n([x-\epsilon,x), x_i )}{n}=\overline{f}(x)-\overline{f}(x-\epsilon)>0.
$$
 This implies that there are infinitely many $x_i \in [x-\epsilon, x)$ for all $\epsilon>0$. Let $\epsilon_n = \frac{1}{n}$. Construct a sequence $x_n$ by choosing $x_n \in [\overline{f}(x-\epsilon_n), \overline{f}(x))$. Thus $x$ is an accumulation point of $\OQmrxp$. The case where $x \in I\left( \underline{f}_{m,r}\right)$ is identical, replacing the $\varlimsup$ with $\liminf$.
\end{proof}

\begin{lem}
If $f:[0,1]\to [0,1]$ is non-decreasing and $x = \inf f^{-1}(y)$, then $x \in I(f)$.
\end{lem}
\begin{proof}
Since $f$ is non-decreasing, if $z<x$, then $f(z)\leq f(x)$. Set $x= \inf f^{-1}(y)$. Then if $f(z)=f(x)$, we have that $f(z)=y$ and $z \in f^{-1}(y)$. But this contradicts $x$ being the infimum of $f^{-1}(y)$. Thus $f(z)<f(x)$. So we have that for all $z<x$, $f(z)<f(x)$, and $x \in I(f)$.
\end{proof}

\begin{proof}[Proof of \reft{thrm3.2}]
Let $x=E_0.E_1E_2\cdots\wrtQ$.
To see that the set of accumulation points of $\OQmrxp$ is exactly $I \left( \overline{f}_{m,r} \right) \cup I\left(\underline{f}_{m,r}\right)$, note that by \refl{acc}, the set of accumulation points of $\OQmrxp$ contains $I \left( \overline{f}_{m,r} \right) \cup I\left(\underline{f}_{m,r}\right)$. For the converse direction, note that by the construction, if $i(n)$ is even, then
$$
\frac{E_{n}}{q_{n}} \in \left[\inf \overline{f}_{i(n)!, n\bmod{i(n)!}}^{-1}\pr{\frac{c(n)}{i(n)}}, \inf \overline{f}_{i(n)!, n\bmod{i(n)!}}^{-1}\pr{\frac{c(n)}{i(n)}}+\Delta_{i(n)}-\frac{1}{q_n}\right].
$$ 
If $i(n)$ is odd, then
$$
\frac{E_{n}}{q_{n}} \in \left[\inf \underline{f}_{i(n)!, n\bmod{i(n)!}}^{-1}\pr{\frac{c(n)}{i(n)}}, \inf \underline{f}_{i(n)!, n\bmod{i(n)!}}^{-1}\pr{\frac{c(n)}{i(n)}}+\Delta_{i(n)}-\frac{1}{q_n}\right].
$$ 
The term $\Delta_{i(n)}-\frac{1}{q_n}$ goes to $0$. Define a sequence $\{ \overline{y}_n \}$ with 
$$
\overline{y}_n = \inf \overline{f}_{i(mn+r)!, {mn+r}\bmod{i(mn+r)!}}^{-1}\left(\frac{c(mn+r)}{i(mn+r)}\right).
$$ 
Define the sequence $\{\underline{y}_n\}$ similarly. Then 
$$
\liminf_{n \to \infty} \left( \frac{E_{mn+r}}{q_{mn+r}}- \overline{y}_n \right) = 0.
$$
 But $\overline{y}_n \in I\left(\overline{f}_{i(mn+r)!, {mn+r}\bmod{i(mn+r)!}}\right)$. Since $\{ \overline{f}_{m,r} \}$ is a linear family of adfs, we have that this set is a subset $I(\overline{f}_{m,r})$, so $\overline{y}_n \in I(\overline{f}_{m,r})$. As $I(\overline{f}_{m,r})$ is closed, all accumulation points of $\overline{y}_n$ must lie in $I(\overline{f}_{m,r})$. The same statements for $\underline{f}_{m,r}$ are true. Furthermore, we have that by construction $\frac{E_{mn+r}}{q_{mn+r}}$ is either arbitrarily close to the sequence $\{ \overline{y}_n\}$ or $\{\underline{y}_n\}$,  so the set of accumulation points of $\OQmrxp$ is contained in the set of accumulation points of $\{\overline{y}_n\} \cup \{ \underline{y}_n \}$.  Thus it is a subset of $I(\overline{f}_{m,r}) \cup I(\underline{f}_{m,r})$.
\end{proof}

\begin{definition}
Let $(X,T)$ be a topological space and $(X,\mathcal{B}(X), \mu)$ be the measure space of X with the Borel $\sigma -$algebra. Then
\begin{equation*}
\supp{\mu} := \{ x\in X : \text{for every neighborhood }N \text{ of }x, \mu(N)>0\}.
\end{equation*}
\end{definition}

\begin{lem}\labl{adfclosed}
For every closed set $D \subseteq [0,1]$, there is an adf $f_D$ so that $I(f_D)=D$.
\end{lem}
\begin{proof}
First we construct a measure $\mu_D$ so that $\supp{\mu_D}=D$. First consider the Borel $\sigma-$algebra of $D$ with the subspace topology. Define 
$$
\mu_D^*(S) = \sup \{ b-a : (a,b)\cap D = (a,b) \cap S \}.
$$
 Extend this measure to the Borel $\sigma-$algebra on $[0,1]$ by $\mu_D(S) = \mu_D^*(S\cap D)$. Then $\supp{\mu_D}=D$, $\mu_D(D)=1$. If $S\cap D = \emptyset$, then $\mu_D(S)=0$. Let $f_D(0)=0$ and $f_D(x) = \mu_D([0,x])$ if $x>0$. Then $f_D(0) = 0$, $f_D(1)=\mu_D([0,1]\cap D) = 1$, and $f_D$ is non-decreasing since $\mu_D([0,x])\leq \mu_D([0,y])$ if $x\leq y$. If $S$ is a set that does not intersect $D$, then $\mu_D(S)=0$ and $f_D(x)=f_D(y)$ for all $x$, $y \in S$. If $x,y \in D$ with $y>x$, then $f_D(y)-f_D(x) = \mu_D([x,y])\geq y-x>0$.
\end{proof}

\begin{proof}[Proof of \reft{strongerWangWenXi}]
We will define adfs $f_{q,s}$ such that if $q \leq k$ we will have $I(f_{q,s}) = D_{q,s}$ and if $q>k$ we will have that $I(f_{q,s}) \subseteq D_{m,r}$ for some $m\leq k$. Applying \reft{thrm3.2} will yield the desired result.

For $q \leq k$, let $A_{q,s} = D_{q,s}$.
For $q >k$, let 
\begin{equation*}
A_{q,r} =  \bigcap_{d | q} D_{d, r\bmod{d}} \cap \bigcap_{d \nmid q} \bigcap_{j=0}^{d-1} D_{d,j}.
\end{equation*}
Define $f_{q,s} = f_{A_{q,s}}$, where $f_D$ is as defined in \refl{adfclosed}.
Then by \reft{thrm3.2}, the sequence $\OQmrxp$ has accumulation points $D_{m,r}$ for $m\leq k$.
\end{proof}

%This is a stronger version of \reft{WangWenXi}.

\subsection{Proof of \reft{NormNotAP}, \reft{OlsenGen}, and \reft{noFrequency}}\labs{OtherProofs}

\begin{proof}[Proof of \reft{NormNotAP}]
The first assertion follows directly from \refl{SameADF}.
For the second, the proof that the set of $Q$-distribution normal numbers is meagre follows identically to the case of the $b$-ary expansion.  The fact that the set of $Q$-distribution normal numbers and the set of numbers $x$ where $\OQmrx$ is u.d. mod $1$ have full measure for all $m>1$ follows by a well known result of H. Weyl.  Thus their difference set has measure zero.

Let $d(n) = \frac{\lcm(1,2,\cdots n)}{\lcm(1,2,\cdots, n-1)}$, $D(n) = \prod_{i=1}^n d(i)$, and
%$\Gamma(n,r) = \{ (r_1, r_2, \cdots, r_{n-1}) \in \mathbb{N}^{n-1} : \text{there is an x so that } x \equiv r_1 \pmod 1, \cdots x \equiv r_{n-1} \pmod n-1, x \equiv r \pmod n \}.$
$\Gamma(n,r) = \{ (r_1, r_2, \cdots, r_{n-1}) \in \mathbb{N}^{n-1} : \exists x \in \mathbb{N} \hbox{ where } x \equiv {r_t}\bmod{t} \ \forall t \in [1,n] \hbox{ and } x \equiv r\bmod{n}\}.$
%x \equiv r_1 \pmod 1, \cdots x \equiv r_{n-1} \pmod n-1, x \equiv r \pmod n \}.
 Set $S_{1,0} = [0,1]$. For all $n$ and $0 \leq r \leq n-1$, set
$$
S_{n,r} = \bigcup_{\vec{r} \in \Gamma(n,r)} \left(\frac{1}{d(n)} \bigcap_{i=1}^{n-1} S_{i,\vec{r}_i} \right)+\frac{r\bmod{d(n)}}{D(n)}.
$$
 For $r \neq r'$, we have that $S_{n,r} \cap S_{n,r'} = \emptyset$ and $\lambda (S_{n,r} ) = \lambda (S_{n,r'} )$. Furthermore, for all $d$, we have that $S_{m,r} = \bigcup_{i=0}^{d-1} S_{md, mi+r}$. Using notation from  \refl{adfclosed}, define $f_{m,r} = f_{S_{m,r}}$.  Thus, we see that $\{ f_{m,r} \}$ is a linear family of adfs. Moreover, $f_{n,r}(x)=x$ for all $x$ only if $n=1$ and $r=0$. Thus the set of real numbers $x$ where $\OQmrx$ has the adf $f_{m,r}$ for all $m$ and $r$ is a subset of the set of numbers that are $Q$-distribution normal but AP $Q$-distribution abnormal and has full Hausdorff dimension by \refmt{Main}.
\end{proof}

%\subsection{No Frequency}

\begin{proof}[Proof of \reft{OlsenGen}]
Define the adf $F$ as follows. Let $c_I, c_J \geq 0$ be constants such that $c_I \lambda(I) = c_J \lambda(J) = x_0$. Since $\lambda(I) + \lambda(J) < 1$, we can find an interval $K$ such that $K \cap (I \cup J)  = \emptyset$ and $\lambda(K) > 0$. Let $c_K \geq 0$ be a constant such that $c_I \lambda(I) + c_J \lambda(J) + c_K \lambda(K) = 1$. Define $$
\phi(x) =
  \begin{cases}
   c_I & \text{if } x \in I \\
   c_J       & \text{if } x \in J \\
   c_K & \text{if } x \in K \\
   0 & \text{if } x \notin I \cup J \cup K
  \end{cases}.$$
Then $\phi$ is integrable,  so we can apply the Lebesgue differentiation theorem to conclude that $F(x) = \int_0^x \phi(t) dt$ is differentiable almost everywhere with $F'(x)~=~\phi(x)$. As $\phi$ is non-negative, we have that $F$ is non-decreasing. Furthermore, $F(0) = 0$. Since $c_A \lambda(A) + c_B \lambda(B) + c_C \lambda(C)  = 1$, we have that $F(1) = 1$. Thus $F$ is an adf. Since $F'(x)=\phi(x)$ almost everywhere, we have for almost every $x \in I$ (respectively $x \in J$) that $F'(x) = c_I$ (respectively $F'(x) = c_J$). Thus if $\omega = \{ x_n \}$ is a sequence with adf $F$, we have that 
\begin{align*}
&\lim_{ n \to \infty} \frac{A_n(I, \omega)}{n} = \int_I F'(x) dx = c_I \lambda(I);\\
&\lim_{ n \to \infty} \frac{A_n(J, \omega)}{n} = \int_J F'(x) dx = c_J \lambda(J). 
\end{align*}
Thus for any sequence with adf $F$, we have that 
$$
f\pr{\lim_{ n \to \infty} \frac{A_n(I, \OQx)}{n}  } = g\pr{\lim_{ n \to \infty} \frac{A_n(J, \OQx)}{n}}.
$$ 
By \refmt{Main}, we have that the set of real numbers $x$ such that $\OQx$ has adf $F$ has full Hausdorff dimension,  so we are done.
\end{proof}

\begin{proof}[Proof of \reft{noFrequency}]
Let $P$ be a set of sparsely intersecting polynomials and
let $\overline{f}_{p,m,r}(x) = x$ and $\underline{f}_{p,m,r}(x) = x^2$ for every $x \in [0,1]$ and $p \in P$. Then $\Phi_{Q,P,F}$ has full Hausdorff dimension.  Moreover, for every $x \in \Phi_{Q,P,F}$, the upper and lower adfs of $\OQfmrx{p}$ are not equal. Thus $\OQfmrx{p}$, and in particular $\OQx$, does not have an adf for all $x$ on a set of full Hausdorff dimension.  The set  $\Phi_{Q,P,F}$ is of zero measure since for almost every $x$ the sequence $\OQx$ is u.d. mod 1.
\end{proof}

\appendix
\section*{Appendix}
I. Niven and H. S. Zuckerman wrote in \cite{NivenZuckerman}:
\begin{quotation}
Let $R$ be a real number with fractional part $.x_1x_2x_3\cdots$ when written to scale $r$.  Let $N(b,n)$ denote the number of occurrences of the digit $b$ in the first $n$ places.  The number $R$ is said to be {\bf simply normal} to scale $r$ if $\lim_{n \to \infty} \frac {N(b,n)} {n}=\frac {1} {r}$  for each of the $r$ possible values of $b$; $R$ is said to be {\bf normal} to scale $r$ if all the numbers $R, rR, r^2R,\cdots$ are simply normal to all the scales $r, r^2, r^3, \cdots$.  These definitions, for $r=10$, we introduced by \'{E}mile Borel \cite{BorelNormal}, who stated (p. 261) that ``la propri\'{e}t\'{e} caract\'{e}ristique'' of a normal number is the following: that for any sequence $B$ whatsoever of $v$ specified digits, we have
\begin{equation}\labeq{quote}
\lim_{n \to \infty} \frac {N(B,n)} {n}=\frac {1} {r^v},
\end{equation}
where $N(B,n)$ stands for the number of occurrences of the sequence $B$ in the first $n$ decimal places
%\cdots
If the number $R$ has the property \refeq{quote} then any sequence of digits
$
B=b_1b_2\cdots b_v
$
appears with the appropriate frequency, but will the frequencies all be the same for $i=1,2,\cdots,v$ if we count only those occurrences of $B$ such that $b_1$ is an $i,i+v,i+2v,\cdots-th$ digit?  It is the purpose of this note to show that this is so, and thus to prove the equivalence of property \refeq{quote} and the definition of normal number.
\end{quotation}
It is not difficult to see how the equivalent definition of normality introduced in \reft{basebnormalii} may be confused with the notion discussed in \reft{basebnormalii}.

%\cite{MyersonPollington}

\section*{Acknowledgment}

Research of the authors is partially supported by the U.S. NSF grant DMS-0943870.  The authors would like to thank Lior Fishman and Mariusz Urbanski for helpful suggestions and reading this manuscript and Vitaly Bergelson for pointing us to the reference \cite{BiluTichy}.  We thank Peter Mueller and Michael Zieve for pointing us to the reference \cite{Tengely} and \url{http://mathoverflow.net/} for providing a forum where we could ask about this.  We are grateful to Kostos Beros for answering our questions on recursion theory.

\bibliographystyle{amsplain}

%\bibliography{mance} 
%\input{CantorSeriesDistribution.bbl}
\providecommand{\bysame}{\leavevmode\hbox to3em{\hrulefill}\thinspace}
\providecommand{\MR}{\relax\ifhmode\unskip\space\fi MR }
\providecommand{\MRhref}[2]{%
  \href{http://www.ams.org/mathscinet-getitem?mr=#1}{#2}
}
\providecommand{\href}[2]{#2}

\end{document}